\documentclass[english]{smfart}

\usepackage{color}
\usepackage{amscd}
\usepackage{amsmath}
\usepackage{ascmac}
\usepackage{cases}
\usepackage{amssymb}
\usepackage{amsfonts}
\usepackage{smfthm}
\usepackage{bull}
\usepackage[all]{xy}
\usepackage{enumerate}
\usepackage[pdftex]{graphicx}
\usepackage[english]{babel}
\theoremstyle{plain}
\newcommand{\cA}{{\mathcal A}}

\renewcommand{\O}{{\mathcal O}}

\newcommand{\cQ}{{\mathcal Q}}

\newcommand\Spec{\mathop{\rm Spec}\nolimits}

\newcommand{\res}{\mathop{\sf res}\nolimits}

\newcommand\lra{\longrightarrow}
\newcommand\ra{\rightarrow}

\title[A family of flat connections and algebraic Garnier solutions]{A family of flat connections on the projective space having dihedral monodromy and algebraic Garnier solutions}
\author{Arata Komyo}
\address{Department of Mathematics, Graduate School of Science, Osaka University, 
Toyonaka, Osaka 560-0043, JAPAN.}
\email{a-koumyou@cr.math.sci.osaka-u.ac.jp}

\begin{document}

\frontmatter

\begin{abstract}
A.\ Girand has constructed an explicit two-parameter family of flat connections 
over the complex projective plane $\mathbb{P}^2$. 
These connections have dihedral monodromy 
and their polar locus is a prescribed quintic composed of a conic and three tangent lines.
In this paper, we give a generalization of this construction.
That is, we construct an explicit $n$-parameter family of flat connections 
over the complex projective space $\mathbb{P}^n$.
Moreover, we discuss the relation between these connections and the Garnier system.
\end{abstract}

\begin{altabstract}
A.\ Girand a construit une famille explicite de connexions \`{a} deux param\`{e}tres
sur le plan projectif complexe $\mathbb{P}^2$.
Ces connexions ont une monodromie di\'{e}drale
et leur lieu polaire est une quintique prescrite compos\'{e}e 
d'une conique et de trois droites tangentes.
Dans cet article, nous donnons une g\'{e}n\'{e}ralisation de cette construction.
Autrement dit, nous construisons une famille de connexions plates explicite
\`a $n$ param\`etres
sur l'espace projectif complexe $\mathbb{P}^n$.
De plus, nous discutons de la relation entre ces connexions et le syst\`{e}me Garnier.
\end{altabstract}


\maketitle

%%%%%%%%%%%%%%%%%%%%%%%%%%%%%%%%%%%%%%%%%%%%%%%%%%
%%%%%%%%%%%%%%%%%%%%%%%%%%%%%%%%%%%%%%%%%%%%%%%%%%%
\section{Introduction}
%%%%%%%%%%%%%%%%%%%%%%%%%%%%%%%%%%%%%%%%%%%%%%%%%%
%%%%%%%%%%%%%%%%%%%%%%%%%%%%%%%%%%%%%%%%%%%%%%%%%%%

A \textit{meromorphic rank $2$ connection} $(E,\nabla)$ on a projective manifold $X$
is tha datum of a rank $2$ vector bundle $E$ equipped with a $\mathbb{C}$-linear 
morphism $\nabla \colon E \ra E \otimes \Omega^1_X (D)$ satisfying 
the Leibniz rule 
\begin{equation*}
\nabla (f \cdot s) = f \cdot \nabla(s) + df \otimes s 
\end{equation*} 
for any section $s$ and function $f$.
Here $D$ is the polar divisor of the connection $\nabla$.
The connection $\nabla$ is \textit{flat} when the curvature vanishes,
that is $\nabla \cdot \nabla=0$.
For a flat meromorphic rank $2$ connection, we can define its monodromy representation.
When $\det(E) = \mathcal{O}_X$ and the trace connection $\mathrm{tr} (\nabla)$
is the trivial connection on $\mathcal{O}_X$,
we say that $(E,\nabla)$ is an \textit{$\mathfrak{sl}_2$-connection}.
A connection $(E,\nabla)$ is called \textit{regular} 
if local $\nabla$-horizontal sections
have moderate growth near the polar divisor $D$
(for details, see \cite[Chap.\ II, Definition 4.2]{Deligne}).

In this paper, we introduce  
a family, parametrized by $\boldsymbol{\lambda} \in \mathbb{C}^n$, 
of meromorphic $\mathfrak{sl}_2$-connections 
$\nabla_{\boldsymbol{\lambda}} = d + A_{\boldsymbol{\lambda}}$ on the trivial bundle 
$\mathcal{O}_{\mathbb{P}^n}\oplus\mathcal{O}_{\mathbb{P}^n}$ 
over $\mathbb{P}^n$ with $n \ge 2$, 
with an explicit connection matrix $A_{\boldsymbol{\lambda}}$.

%%%%%%%%%%%%%%%%%%%%%%%%%%%%%%%%%%%%%%%%%%%%%%
\subsection{The explicit expression of $\nabla_{\boldsymbol{\lambda}}$}\label{2019.7.10.22.19}
%%%%%%%%%%%%%%%%%%%%%%%%%%%%%%%%%%%%%%%%%%%%%%
The explicit connection matrix $A_{\boldsymbol{\lambda}}$ is described as follows.
Let $[x:y:z_1:\ldots:z_{n-2}:t]$ be the homogeneous coordinates of $\mathbb{P}^n$.
Set $f(x,y,t):=x^2 + y^2 +t^2  -2 (x y +yt   +  t x)$.
For $\boldsymbol{\lambda}= (\lambda_0 , \ldots, \lambda_{n-1}) \in \mathbb{C}^n$,
we define rational $1$-forms on $\mathbb{P}^n$ as follows:
\begin{equation*}
\begin{aligned}
\alpha_0(x,y) &:= - 
\frac{(2 \lambda_0 + \lambda_1)dx - (2 \lambda_1 + \lambda_0)dy}{2} -
\frac{ \lambda_1 ( y-1)}{2}\frac{ dx}{x} +
\frac{ \lambda_0 ( x-1)}{2}\frac{ dy}{y}  ,\\
\alpha_1(x,y) &:= - \frac{1}{4}\frac{ d f(x,y,1)}{f(x,y,1)}, \quad
\alpha_2(x,y):= -\frac{\alpha_0(x,y)}{f(x,y,1)}, 
\end{aligned}
\end{equation*}
and
\begin{equation*}
\begin{aligned}
\alpha^i_0(x,y, z_i) &:=\lambda_{i+1} \left(  d z_i-
\frac{ z_id (f(x,y,1)-z_i^2) }{2(f(x,y,1)-z_i^2)} \right), \quad 
\alpha^i_2(x,y, z_i):= -\frac{\alpha^i_0(x,y, z_i)}{f(x,y,1)}  ,
\end{aligned}
\end{equation*}
which are described by the affine coordinates $[x:y:z_1:\ldots:z_{n-2}:1]$.
We define a connection matrix $A_{\boldsymbol{\lambda}}$ as 
\begin{equation*}
\begin{aligned}
A_{\boldsymbol{\lambda}}= 
\begin{pmatrix}
\cA_{11}  & \cA_{12}  \\
-\cA_{21}  &- \cA_{11} 
\end{pmatrix} + \sum_{i=1}^{n-2}
\begin{pmatrix}
\cA_{11}^i  & \cA_{12}^i  \\
-\cA_{21}^i &- \cA_{11}^i 
\end{pmatrix},
\end{aligned}
\end{equation*}
\textit{where}
\begin{equation*}
\begin{aligned}
\cA_{11} &:=
 (x-1)  \alpha_2 (x,y)  
 +  \alpha_1(x,y) +\frac{1}{2}\frac{d y}{y} ,\\
\cA_{11}^i &:=  (x-1)  \alpha^i_2 (x,y, z_i), \\
\cA_{12}  &:= \frac{dx    + (x-1)^2\alpha_2 (x,y)
 + 2 (x -1)\alpha_1(x,y)  + \alpha_0(x,y)  }{y} , \\
\cA_{12}^i &:=
 \frac{(x-1)^2\alpha^i_2 (x,y, z_i) + \alpha^{i}_0(x,y, z_i)  }{y},\\
\cA_{21}&:=  y  \alpha_2 (x,y),  \\
\cA_{21}^i &:= y \alpha^i_2 (x,y, z_i)  ,
\end{aligned}
\end{equation*}
in the affine coordinates $[x:y:z_1:\ldots:z_{n-2}:1]$.

%%%%%%%%%%%%%%%%%%%%%%%%%%%%%%%%%%%%%%%%%%%%%%%%
\subsection{Main results}
%%%%%%%%%%%%%%%%%%%%%%%%%%%%%%%%%%%%%%%%%%%%%%%%

Let $\mathcal{Q}_0$ and $\mathcal{Q}_i$ be the divisors on $\mathbb{P}^n$ defined by
$\mathcal{Q}_0:= ( f (x, y ,t) =0 )$ and
$\mathcal{Q}_i:= ( f (x, y ,t)  - z_i^2  =0 )$, $i=1,\ldots,n-2$,
respectively.
Let $D_n$ be the divisor on $\mathbb{P}^n$ defined by
\begin{equation*}
D_n := ( x=0 ) + ( y=0 ) + ( t=0 ) +
\mathcal{Q}_0 + \mathcal{Q}_1 + \cdots + \mathcal{Q}_{n-2}.
\end{equation*}
From the explicit expression of $\nabla_{\boldsymbol{\lambda}}$, it follows that 
all $\nabla_{\boldsymbol{\lambda}}$ share the same polar divisor $D_n$.
Note that the conic $\mathcal{Q}_0$ plays a special role: 
it is tangent to the conic $\mathcal{Q}_i$ inside the coordinate hyperplane $(z_i=0)$ 
for $i=1,\ldots,n-2$,
and it is tangent to the three coordinate hyperplanes $(x=0)$, $(y=0)$ and $(t=0)$.

%%%%%%%%%%%%%%%%%%%%%%%%%%%%%%%%%%%%%%%%%%%%%%%%
\begin{theo}\label{2019.7.10.21.43}
\textit{
For each $\boldsymbol{\lambda}$, the connection $\nabla_{\boldsymbol{\lambda}}$
is flat and has at worst regular singularities. 
}
\end{theo}
%%%%%%%%%%%%%%%%%%%%%%%%%%%%%%%%%%%%%%%%%%%%%%%%

We say that two connections $(E,\nabla)$ and $(E',\nabla')$ 
are \textit{birationally equivalent} when there is a birational bundle transformation 
$\phi \colon E \dashrightarrow E'$ that conjugates the two operators $\nabla$ 
and $\nabla'$.
We say that two connections $(E,\nabla)$ and $(E',\nabla')$ 
are \textit{projectively equivalent} if the induced $\mathbb{P}^1$-bundles 
coincide $\mathbb{P}(E)=\mathbb{P}(E')$, and if moreover $\nabla$ and $\nabla'$
induce the same projective connection $\mathbb{P}(\nabla)= \mathbb{P}(\nabla')$.

%%%%%%%%%%%%%%%%%%%%%%%%%%%%%%%%%%%%%%%%%%%%%%%%
\begin{theo}\label{2019.7.10.15.33}
\textit{Via a generically finite Galois morphism 
$f \colon \mathbb{P}^n \rightarrow \mathbb{P}^n$, for each $\boldsymbol{\lambda}$,
the pull-back connection $f^*\nabla_{\boldsymbol{\lambda}}$ on the trivial bundle
is projectively birationally equivalent to a split connection of the form 
\begin{equation*}
d+ \begin{pmatrix}
\omega & 0 \\
0 &-\omega
\end{pmatrix}
\end{equation*}
with $\omega$ a rational closed $1$-form on $X$.
}
\end{theo}
%%%%%%%%%%%%%%%%%%%%%%%%%%%%%%%%%%%%%%%%%%%%%%%%

In this case, the generically finite Galois morphism is a genetically finite morphism of degree two.
Loray, Pereira, and Touzet have proved
the structure theorem of flat meromorphic $\mathfrak{sl}_2$-connections
on projective manifolds in \cite{LPT}
(see also \cite{CS}).
By Theorem \ref{2019.7.10.15.33}, each $\nabla_{\boldsymbol{\lambda}}$ is the first type of 
the three possible types of flat meromorphic $\mathfrak{sl}_2$-connections 
over projective manifolds in the sense of Loray, Pereira, and Touzet \cite[Theorem E]{LPT}.
%Here these three types are not mutually exclusive.

Since the connection $\nabla_{\boldsymbol{\lambda}}$ is flat for each $\boldsymbol{\lambda}$,
we can define its monodromy representation
$\pi_1 (\mathbb{P}^n \setminus D_n) 
\rightarrow 
\mathrm{SL}_2(\mathbb{C})$
of $\nabla_{\boldsymbol{\lambda}}$ for each $\boldsymbol{\lambda}$. 
Let $\boldsymbol{D}_{\infty}$ be the infinite dihedral group:
\begin{equation*}
\boldsymbol{D}_{\infty}:= \left\langle
\begin{pmatrix}
0 & \alpha \\
-\alpha^{-1} & 0
\end{pmatrix},
\begin{pmatrix}
\beta & 0 \\
0 & \beta^{-1}
\end{pmatrix}\ 
\middle| \ 
\alpha , \beta \in \mathbb{C}^*
 \right\rangle \le \mathrm{SL}_2 (\mathbb{C}).
\end{equation*}
For the monodromy representation of $\nabla_{\boldsymbol{\lambda}}$,
we have the following.

%%%%%%%%%%%%%%%%%%%%%%%%%%%%%%%%%%%%%%%%%%%%%%%%
\begin{theo}\label{2019.7.10.22.03}
\textit{
For generic $\boldsymbol{\lambda}$, 
the monodromy representation of $\nabla_{\boldsymbol{\lambda}}$ is conjugated to an explicit
representation 
\begin{equation*}
\rho_{\boldsymbol{\lambda}} \colon \pi_1 (\mathbb{P}^n \setminus D_n) 
\longrightarrow 
\mathrm{SL}_2(\mathbb{C}),
\end{equation*}
which is virtually abelian, i.e. abelian after a finite cover of $\mathbb{P}^n \setminus D_n$,
and takes values in the infinite dihedral group $\boldsymbol{D}_{\infty}$.
}
\end{theo}
%%%%%%%%%%%%%%%%%%%%%%%%%%%%%%%%%%%%%%%%%%%%%%%%

%%%%%%%%%%%%%%%%%%%%%%%%%%%%%%%%%%%%%%%%%%%%%%%%
\subsection{Algebraic Garnier solution}
%%%%%%%%%%%%%%%%%%%%%%%%%%%%%%%%%%%%%%%%%%%%%%%%

The $(2n-2)$-variable \textit{Garnier system} $\mathcal{G}_{2n-2}$ 
is the completely integrable Hamiltonian system
\begin{equation*}
\mathcal{G}_{2n-2} \colon
\left\{
\begin{aligned}
\frac{\partial \rho_j}{\partial t_i} &= -\frac{\partial K_i}{\partial \nu_j} &
i,j= 1,\ldots, 2n-2 \\
\frac{\partial \nu_j}{\partial t_i} &= \frac{\partial K_i}{\partial \rho_j} &
i,j= 1,\ldots, 2n-2,
\end{aligned}
\right.
\end{equation*}
where
\begin{equation*}
K_i= -\frac{\Lambda(t_i)}{T'(t_i)}
\left[ \sum^{2n-2}_{k=1} \frac{T(\nu_k)}{(\nu_k-t_i) \Lambda'(\nu_k)}
\left\{ \rho_k^2 -\sum^{2n}_{m=1} \frac{\theta_m-\delta_{im}}{\nu_k - t_m} \rho_k
+ \frac{\kappa}{\nu_k(\nu_k -1)} \right\}
\right]
\end{equation*}
with $t_{2n-1}=0$, $t_{2n}=1$, 
$\kappa:= \frac{1}{4}\left\{ (\sum^{2n}_{m=1} \theta_m -1)^2 - (\theta_{\infty}^2 +1) \right\}$,
$\Lambda(t):= \prod^{2n-2}_{k=1} (t-\nu_k)$ and 
$T(t):= \prod^{2n}_{k=1} (t-t_k)$
(see \cite{G1}, \cite{G2}, and \cite{Okamoto}).
Here $\theta_m$ ($m=1,\ldots,2n,\infty$) is the constant parameters defined by
\begin{equation*}
\begin{aligned}
\theta_1&= \frac{1}{2}, & \theta_2&= \frac{1}{2}, &
\theta_{2i+1}&= \lambda_{i+1},&
\theta_{2i+2}&= \lambda_{i+1}\ (i=1,\ldots,n-2) \\
\theta_{2n-1}&= \lambda_{1},&
\theta_{2n}&= \lambda_{0}-1,&
\theta_{\infty}&= \lambda_{0}+\lambda_{1}.
\end{aligned}
\end{equation*}

To give a solution of the Garnier system $\mathcal{G}_{2n-2}$,
we consider the Fuchsian system with $2n + 1$ regular singularities 
at $0,1,t_1 , \ldots , t_{2n-2}, \infty$:
\begin{equation}\label{2018.5.22.12.11}
 d+ 
\tilde{H}_{2n-1}\frac{d\tilde{x}}{\tilde{x}} + 
\tilde{H}_{2n}\frac{d\tilde{x}}{\tilde{x}-1} +
\sum_{i=1}^{2n-2} \tilde{H}_{i}\frac{d\tilde{x}}{\tilde{x}-t_i} ,
\end{equation}
where $\tilde{H}_{i}$ ($i=1, \ldots , 2n$) are 
$2\times 2$ matrices independent of $\tilde{x}$ and 
$t_i \neq t_j$ ($i \neq j$). 
We assume that  $\tilde{H}_{2n+1}:=-\sum_{i=1}^{2n} \tilde{H}_{i}$ is a diagonal matrix and 
the eigenvalues of $\tilde{H}_{i}$ ($i=1, \ldots , 2n+1$) are
as in Table \ref{2018.5.17.17.04}.

\begin{table}[htb]
\caption{The eigenvalues of the residue matrices 
($i=1,\ldots , n-2$).}
  \begin{tabular}{c|ccccccc}\label{2018.5.17.17.04}
  Reside matrices  & $\tilde{H}_{1}$ &
    $\tilde{H}_{2}$ & $\tilde{H}_{2i+1}$ & $\tilde{H}_{2i+2}$
    &  $\tilde{H}_{2n-1}$ &  $\tilde{H}_{2n}$ & $\tilde{H}_{2n+1}$ \\\hline
  Eigenvalues 
     & $\pm \frac{1}{4}$ 
      & $\pm \frac{1}{4}$ 
      & $\pm \frac{\lambda_{i+1}}{2}$ 
      & $\pm \frac{\lambda_{i+1}}{2}$ 
       & $\pm \frac{\lambda_1}{2}$  
    & $\pm \frac{\lambda_0-1}{2}$ 
      &  $\pm\frac{\lambda_0+\lambda_1}{2}$
  \end{tabular}
\end{table}

We fix generators $\gamma_{\tilde{x}}$ 
($\tilde{x}=0,1, t_1 \ldots ,t_{2n-2}, \infty$)
 of the fundamental group $\pi_1(\mathbb{P}^1 \setminus \{ 0,1,t_1,\ldots,t_{2n-2}, \infty \},*)$.
Here the loop $\gamma_{\tilde{x}}$ on $\mathbb{P}^1$ 
is oriented counter-clockwise, 
$\tilde{x}$ lies inside, while the other singular points lie outside.
Let $\rho'_{\boldsymbol{\lambda}}\colon 
\pi_1(\mathbb{P}^1 \setminus\{ t_1,\ldots,t_{2n}, \infty \},*) \ra \mathrm{SL}_2(\mathbb{C})$
be the representation of the fundamental group defined by Table \ref{2018.5.14.12.41}.
If we have the isomonodromic deformation of the Fuchsian system (\ref{2018.5.22.12.11})
whose preserved monodromy representation is conjugated to
 $\rho'_{\boldsymbol{\lambda}}$,
 then we obtain a solution of the Garnier system $\mathcal{G}_{2n-2}$ 
 (see \cite[Section 2]{Mazz}).

\begin{table}[htb]
\caption{The representation of the fundamental group;
here $a_j = \exp(-\pi \sqrt{-1} \lambda_j)$ $j=0,1,\ldots,n-1$.}
\begin{center}
  \begin{tabular}{c|c|c|c}\label{2018.5.14.12.41}
     $\gamma_0$ &  $\gamma_1$ & $\gamma_{t_1}$ &
    $\gamma_{t_2}$ \\\hline
    $\begin{pmatrix} a_1 & 0 \\ 0 & a_1^{-1} \end{pmatrix}$  
    & $\begin{pmatrix} -a_0 & 0 \\ 0 & -a_0^{-1} \end{pmatrix}$ 
     & $\begin{pmatrix} 0 & 1 \\ -1 &0 \end{pmatrix}$ 
      & $\begin{pmatrix} 0 & a_0^{2} \\ -a_0^{-2} &0 \end{pmatrix}$ 
  \end{tabular}\\
  \begin{tabular}{c|c|c}
     $\gamma_{t_{2i+1}}$  ($i=1,\ldots,n-2$) &  $\gamma_{t_{2i+2}}$  ($i=1,\ldots,n-2$)
      & $\gamma_{\infty}$\\\hline
    $\begin{pmatrix} a_{i+1} & 0 \\ 0 & a_{i+1}^{-1} \end{pmatrix}$  
    & $\begin{pmatrix} a^{-1}_{i+1} & 0 \\ 0 & a_{i+1} \end{pmatrix}$ 
     & $\begin{pmatrix} a_0 a_1^{-1} & 0 \\ 0 &a_0^{-1} a_1 \end{pmatrix}$ 
  \end{tabular}
  \end{center}
\end{table}

We say $(\rho_j(t_1,\ldots,t_{2n-2}), \nu_j(t_1,\ldots,t_{2n-2}))_{j=1,...,2n-2} $
is an \textit{algebraic solution of $\mathcal{G}_{2n-2}$}
if $(\rho_j(t_1,\ldots,t_{2n-2}), \nu_j(t_1,\ldots,t_{2n-2}))_{j=1,...,2n-2} $ 
satisfies the Garnier system $\mathcal{G}_{2n-2}$
and the graph of the solution has Zariski closure of dimension $2n-2$.

%%%%%%%%%%%%%%%%%%%%%%%%%%%%%%%%%%%%%%%%%%%%%%%%
\begin{theo}\label{2019.7.10.22.13}
Let $T$ be a certain Zariski open subset of $\mathbb{A}^{2n-2}$
parametrizing generic lines in $\mathbb{P}^n$.
From the natural morphism $\mathbb{P}^1 \times T \rightarrow \mathbb{P}^n$,
one obtains a relative connection $(\nabla_{\mathbb{P}^1 \times T/T})_{\boldsymbol{\lambda}}$
with $2n+1$ simple poles by the pull-back of $\nabla_{\boldsymbol{\lambda}}$.
\begin{itemize}
\item[(i)] Up to an \'etale base change $\tilde{T} \rightarrow T$, an isomorphism of the 
relative trivial bundle, and up to relative M\"obius transformations in the base, 
we can consider
the relative connection $(\nabla_{\mathbb{P}^1 \times T/T})_{\boldsymbol{\lambda}}$
as a family of the Fuchsian system (\ref{2018.5.22.12.11})
parametrized by $T$.
\item[(ii)] The family $(\nabla_{\mathbb{P}^1 \times T/T})_{\boldsymbol{\lambda}}$
is isomonodromic. 
The preserved monodromy representation of the fundamental group 
$\pi_1(\mathbb{P}^1 \setminus \{ 0,1,t_1,\ldots,t_{2n-2}, \infty \},*)$
of this isomonodromic family
 is conjugated to the representation given by Table \ref{2018.5.14.12.41}
 \item[(iii)] Since $\dim \tilde{T}=2n-2$, the connection matrices 
 of the isomonodromic family $(\nabla_{\mathbb{P}^1 \times T/T})_{\boldsymbol{\lambda}}$
defines an algebraic solution of the Garnier system $\mathcal{G}_{2n-2}$.
\end{itemize}
\end{theo}
%%%%%%%%%%%%%%%%%%%%%%%%%%%%%%%%%%%%%%%%%%%%%%%%

In the case of $n = 2$, 
the family of connections $\nabla_{\boldsymbol{\lambda}}$ 
have been established by Girand in \cite{Girand}.
Moreover Girand have discussed 
an explicit relation to certain algebraic solutions of the sixth Painlev\'e equation in \cite{Girand}.
Our argument is generalization of 
Girand's idea of explicit construction of $\nabla_{\boldsymbol{\lambda}}$, 
and of the proof of the main results, to the case $n \ge 2$.

The organization of this paper is as follows.
In Section \ref{2018.5.18.17.10}, we introduce  
a family, parametrized by $\boldsymbol{\lambda} \in \mathbb{C}^n$, 
of meromorphic $\mathfrak{sl}_2$-connections 
$\nabla_{\boldsymbol{\lambda}} = d + A_{\boldsymbol{\lambda}}$ on the trivial bundle 
$\mathcal{O}_{\mathbb{P}^n}\oplus\mathcal{O}_{\mathbb{P}^n}$ 
over $\mathbb{P}^n$ with $n \ge 2$, 
with an explicit connection matrix $A_{\boldsymbol{\lambda}}$.
In Section \ref{2019.7.10.22.00},
we show Theorem \ref{2019.7.10.21.43} and Theorem \ref{2019.7.10.15.33}.
In Section \ref{2018.5.18.17.14}, we compute the monodromy representation of  
$\nabla_{\boldsymbol{\lambda}}$ for generic $\boldsymbol{\lambda}$.
In Section \ref{2019.7.10.22.02},
we show Theorem \ref{2019.7.10.22.03}.
In Section \ref{2018.5.18.17.17}, 
we consider the natural morphism $\mathbb{P}^1 \times T \rightarrow \mathbb{P}^n$,
where $T$ is a certain Zariski open subset of $\mathbb{A}^{2n-2}$
parametrizing generic lines in $\mathbb{P}^n$.
Let $(\nabla_{\mathbb{P}^1 \times T/T})_{\boldsymbol{\lambda}}$ be 
the relative connection with $2n+1$ simple poles
given by the pull-back of $\nabla_{\boldsymbol{\lambda}}$.
In Section \ref{2019.7.10.22.11}, 
we introduce an \'etale base change $\tilde{T} \rightarrow T$
to prove the assertion (i) of Theorem \ref{2019.7.10.22.13}.
In Section \ref{2019.7.12.11.27},  
after the \'etale base change $\tilde{T} \rightarrow T$,
we compute the residue matrix of 
$(\nabla_{\mathbb{P}^1 \times \tilde{T}/\tilde{T}})_{\boldsymbol{\lambda}}$ 
for each simple pole.
In Section \ref{2019.7.10.22.15}, 
we recall the relation between isomonodromic deformations and the Garnier system
following \cite{Mazz}.
In Section \ref{2019.7.12.11.28},
we show Theorem \ref{2019.7.10.22.13}.

%%%%%%%%%%%%%%%%%%%%%%%%%%%%%%%%%%%%%%%%%%%%%%%%%%
%%%%%%%%%%%%%%%%%%%%%%%%%%%%%%%%%%%%%%%%%%%%%%%%%%
\section{Construction of flat connections on projective spaces}\label{2018.5.18.17.10}
%%%%%%%%%%%%%%%%%%%%%%%%%%%%%%%%%%%%%%%%%%%%%%%%%%
%%%%%%%%%%%%%%%%%%%%%%%%%%%%%%%%%%%%%%%%%%%%%%%%%%%

In this section,
we introduce  
a family, parametrized by $\boldsymbol{\lambda} \in \mathbb{C}^n$, 
of meromorphic $\mathfrak{sl}_2$-connections 
$\nabla_{\boldsymbol{\lambda}} = d + A_{\boldsymbol{\lambda}}$ on the trivial bundle 
$\mathcal{O}_{\mathbb{P}^n}\oplus\mathcal{O}_{\mathbb{P}^n}$ 
over $\mathbb{P}^n$ with $n \ge 2$, 
with the explicit connection matrix $A_{\boldsymbol{\lambda}}$ 
described in Section \ref{2019.7.10.22.19}.
For this introduction, we start from
a family, parametrized by $\boldsymbol{\lambda} \in \mathbb{C}^n$, 
of flat meromorphic $\mathfrak{sl}_2$-connections $(\nabla_0)_{\boldsymbol{\lambda}}$ 
on the trivial bundle 
$\mathcal{O}_{\mathbb{C}^n}\oplus\mathcal{O}_{\mathbb{C}^n}$ 
over $\mathbb{C}^n$
whose connection matrix splits.
Next, we consider a birational transformation of the projective connection 
$\mathbb{P}((\nabla_0)_{\boldsymbol{\lambda}})$.
We define a generically finite Galois morphism 
$f \colon \mathbb{C}^n \rightarrow \mathbb{C}^n$.
We show that this birational transformation descend to
 a projective connection over $\mathbb{C}^n$.
We denote by $\mathbb{P}((\nabla_1)_{\boldsymbol{\lambda}})$ this projective connection.
The connection corresponding to $\mathbb{P}((\nabla_1)_{\boldsymbol{\lambda}})$
does not split. 
If we extend the projective connection $\mathbb{P}((\nabla_1)_{\boldsymbol{\lambda}})$ 
over $\mathbb{C}^n$
to a projective connection over $\mathbb{P}^n$ naively,
then the extended projective connection over $\mathbb{P}^n$ has poles of oder $2$ along
the divisor $\mathbb{P}^n \setminus \mathbb{C}^n$.
Then we consider a birational transformation 
of $\mathbb{P}((\nabla_1)_{\boldsymbol{\lambda}})$.
By this birational transformation, we obtain 
the meromorphic $\mathfrak{sl}_2$-connections 
$\nabla_{\boldsymbol{\lambda}} = d + A_{\boldsymbol{\lambda}}$
with the explicit connection matrix $A_{\boldsymbol{\lambda}}$ 
described in Section \ref{2019.7.10.22.19}.
Finally Theorem \ref{2019.7.10.21.43} and Theorem \ref{2019.7.10.15.33} follow from 
this construction of $\nabla_{\boldsymbol{\lambda}}$.

%%%%%%%%%%%%%%%%%%%%%%%%%%%%%%%%%%%%%%%%%%%%%%%%
\subsection{Flat connections $(\nabla_0)_{\boldsymbol{\lambda}}$ defined by
rational closed 1-forms}
%%%%%%%%%%%%%%%%%%%%%%%%%%%%%%%%%%%%%%%%%%%%%%%%
Let $\lambda_0, \ldots, \lambda_{n-1} $ be complex numbers.
Set $Y:=\Spec \mathbb{C} [u_0,u_1,z_1,\ldots,z_{n-2}]$.
Let $\omega_0$ and $\psi_n$ be the closed rational 1-forms on $Y$ defined by
\begin{equation*}
\begin{aligned}
\omega_0 :=&\ \lambda_0 \left( \frac{du_0}{u_0} - \frac{du_1}{u_1}\right) 
+ \lambda_1 \left( \frac{du_0}{u_0-1} - \frac{du_1}{u_1-1}\right)\\
\psi_n :=&\ 
\begin{cases}
\sum_{i=1}^{n-2} \lambda_{i+1} \left(
\frac{  d(u_0-u_1 + z_i)}{u_0-u_1 + z_i}
-\frac{d (u_0-u_1 - z_i)}{u_0-u_1 - z_i} \right) & n>2\\
0 & n=2.
\end{cases} 
\end{aligned}
\end{equation*}
We have a family of flat connections
\begin{equation*}
(\nabla_0)_{\boldsymbol{\lambda}} :=d
+\frac{1}{2}
\begin{pmatrix}
\omega_0 + \psi_n & 0 \\
0 & -\omega_0 - \psi_n\\
\end{pmatrix}
\end{equation*}
on the trivial rank 2 vector bundle $E_0 \ra Y$.
The family $(\nabla_0)_{\boldsymbol{\lambda}}$ 
is parametrized by $\boldsymbol{\lambda}=(\lambda_0,\ldots,\lambda_{n-1})$.
On the associated projective bundle $\mathbb{P}(E_0)$,
we have the associated projective connection
$\mathbb{P}((\nabla_0)_{\boldsymbol{\lambda}})=d w_0 + (\omega_0 + \psi_n ) w_0$,
where $w_0$ is a projective coordinate on the fibers.

%%%%%%%%%%%%%%%%%%%%%%%%%%%%%%%%%%%%%%%%%%%%%%%%
\subsection{Descent of the connection $(\nabla_0)_{\boldsymbol{\lambda}}$}\label{2019.7.10.23.01}
%%%%%%%%%%%%%%%%%%%%%%%%%%%%%%%%%%%%%%%%%%%%%%%%
We consider the birational transformation of the projective connection 
$\mathbb{P}((\nabla_0)_{\boldsymbol{\lambda}})$ defined by
 $\Phi \colon \mathbb{P}(E_0) \dashrightarrow \mathbb{P}(E_0)$;
\begin{equation*}
\begin{aligned}
(u_0,u_1,z_1,\ldots,z_{n-2}, [w_0^0:w_0^1]) 
\longmapsto (u_0,u_1,z_1,\ldots,z_{n-2},[\tilde{w}_0^0:\tilde{w}_0^1] ),
\end{aligned}
\end{equation*}
where 
\begin{equation}\label{2018.5.4.11.20}
\begin{aligned}
\frac{\tilde{w}_0^1}{\tilde{w}_0^0}=
(u_0 - u_1) \frac{w^1_0+w^0_0}{w^1_0-w_0^0} .
\end{aligned}
\end{equation}
The rational function (\ref{2018.5.4.11.20}) is an invariant 
of the involution $\iota\colon \mathbb{P}(E_0) \ra \mathbb{P}(E_0)$;
\begin{equation*}
\begin{aligned}
\iota\colon
(u_0,u_1,z_1,\ldots,z_{n-2}, [w_0^0:w_0^1]) 
\longmapsto (u_1,u_0,z_1,\ldots,z_{n-2},[w_0^1:w_0^0] ),
\end{aligned}
\end{equation*}
that is $(\tilde{w}_0^1/\tilde{w}_0^0) \circ \iota=\tilde{w}_0^1/\tilde{w}_0^0$ as 
functions on $\mathbb{P}(E_0)$.
Put $w_0=w_0^1/w_0^0$ and $\tilde{w}_0=\tilde{w}_0^1/\tilde{w}_0^0$. 
We can check the following proposition by direct computation.
%%%%%%%%%%%%%%%%%%%%%%%%%%%%%%%%%%%%%%%%%%%%%%%%%
\begin{prop}
We define a map $f \colon Y \ra \mathbb{P}^n$ by
\begin{equation*}
\begin{aligned}
(u_0,u_1,z_1, \dots,z_{n-2}) 
&\longmapsto [s_1:s_2 : z_1 : \ldots :z_{n-2}:1 ],
\end{aligned}
\end{equation*}
where $s_1=u_0+u_1$ and $s_2=u_0 u_1$.
The birational transformation $(\Phi^{-1})^* \mathbb{P}((\nabla_0)_{\boldsymbol{\lambda}})$
on $\mathbb{P}(E_0)$ descends to a projective connection on
$f(Y) \times \mathbb{P}^1 \ra f(Y)$:
\begin{equation}\label{2018.4.27.14.06}
\begin{aligned}
(\Phi^{-1})^* \mathbb{P}((\nabla_0)_{\boldsymbol{\lambda}})
&=\frac{d \tilde{w}_0}{d w_0} \left( dw_0 + (\omega_0 + \psi_n) w_0 \right)  \\
&= 
d \tilde{w}_0 + \left( \alpha_2 (s_1,s_2) + \sum^{n-2}_{i=1} 
\alpha^i_2 (s_1,s_2, z_i) \right) \tilde{w}_0^2 \\
&\quad + 2 \alpha_1(s_1,s_2) \tilde{w}_0 
+ \left(\alpha_0 ( s_1,s_2) + \sum^{n-2}_{i=1} 
\alpha^i_0 (s_1,s_2, z_i)\right),
\end{aligned}
\end{equation}
where
\begin{equation}\label{2018.4.27.14.04}
\begin{aligned}
\alpha_0 ( s_1,s_2) &:= 
\frac{2 \lambda_0 (1  - s_1 + s_2 ) +  \lambda_1 (-s_1 + 2 s_2)}{2( 1-  s_1 +  s_2)} d s_1 
 -
\frac{ \lambda_0 s_1 ( 1 - s_1  + s_2) +\lambda_1 s_2 (s_1-2)}{2 s_2 (1  - s_1 + s_2)}  ds_2, \\
\alpha^i_0 ( s_1,s_2,z_i)&:= \lambda_{i+1} \left(  d z_i
-
\frac{ z_id (s_1^2-4 s_2 - z_i^2) }{2(s_1^2-4 s_2 - z_i^2)} \right) ,\\
\alpha_1(s_1,s_2)
 &:=-\frac{1}{4}\frac{d(s_1^2 - 4 s_2)}{s_1^2 - 4 s_2}, \\
\alpha_2(s_1,s_2)
& := - \frac{\alpha_0(s_1,s_2)}{s_1^2 -4 s_2} , \\
\alpha^i_2(s_1,s_2, z_i) &:=-    \frac{\alpha^i_0(s_1,s_2, z_i)}{s_1^2-4 s_2}.
\end{aligned}
\end{equation}
\end{prop}
%%%%%%%%%%%%%%%%%%%%%%%%%%%%%%%%%%%%%%%%%%%%%%%%%%
The corresponding connection $(\nabla_1)_{\boldsymbol{\lambda}}$
on $f(Y) \times \mathbb{C}^2 \ra f(Y)$
 is
\begin{equation*}
(\nabla_1)_{\boldsymbol{\lambda}}=
d+ \begin{pmatrix}
\alpha_1 (s_1,s_2) & \alpha_0 (s_1,s_2)\\
-\alpha_2 (s_1,s_2) & -\alpha_1 (s_1,s_2)
\end{pmatrix}+\sum^{n-1}_{i=1} 
\begin{pmatrix}
0 & \alpha_0^i (s_1,s_2,z_i)\\
-\alpha_2^i (s_1,s_2,z_i) & 0
\end{pmatrix}.
\end{equation*}
We consider a relation between 
this connection and the connection $(\nabla_0)_{\boldsymbol{\lambda}}$.
Let $\nabla_0'$ be the meromorphic connection on $Y \times \mathbb{C}\ra Y$
defined by
$\nabla_0':=d-\frac{1}{2} 
\frac{d(u_0-u_1)}{u_0-u_1}$.
We define a matrix $M_1(u_0,u_1)$ on $Y$ by
\begin{equation*}
M_1(u_0,u_1):=
\begin{pmatrix}
-1 & -u_0+ u_1\\
-1 & u_0- u_1
\end{pmatrix}.
\end{equation*}
Let $\nabla_0''$ be the meromorphic connection on $Y \times \mathbb{C}^2\ra Y$
defined by
\begin{equation*}
\begin{aligned}
\nabla_0''&:=d+M_1(u_0,u_1)^{-1} dM_1(u_0,u_1) \\
&\qquad +
M_1(u_0,u_1)^{-1}
\frac{1}{2}
\begin{pmatrix}
\omega_0 + \psi_n & 0 \\
0 & -\omega_0 - \psi_n\\
\end{pmatrix}
M_1(u_0,u_1).
\end{aligned}
\end{equation*}
Then we have 
\begin{equation}\label{2018.5.8.23.53}
f^* (\nabla_1)_{\boldsymbol{\lambda}} = \nabla_0'' \otimes \nabla_0'.
\end{equation}

Moreover, we consider the map $\mathbb{P}^n \ra \mathbb{P}^n$;
$[s_1:s_2:z_1:\ldots:z_{n-2}:t] \mapsto [x:y:z_1:\ldots :z_{n-2}:t]$,
where $x:=t - s_1 +s_2$ and $y:=s_2$.
Set
\begin{equation*}
\begin{aligned}
f(x,y) :=&\ x^2 + y^2 +1  -2 (x y + x + y)  .
\end{aligned} 
\end{equation*}
Then the rational 1-forms (\ref{2018.4.27.14.04}) are transformed into 
\begin{equation}\label{2018.4.27.13.15}
\begin{aligned}
\alpha_0(x,y) &= - 
\frac{(2 \lambda_0 + \lambda_1)dx - (2 \lambda_1 + \lambda_0)dy}{2} -
\frac{ \lambda_1 ( y-1)}{2}\frac{ dx}{x} +
\frac{ \lambda_0 ( x-1)}{2}\frac{ dy}{y} , \\
\alpha^i_0(x,y, z_i) &=\lambda_{i+1} \left(  d z_i
-
\frac{ z_id (f(x,y)-z_i^2) }{2(f(x,y)-z_i^2)} \right), \\
\alpha_1(x,y) &= - \frac{1}{4}\frac{ d f(x,y)}{f(x,y)},\\
\alpha_2(x,y)&= -\frac{\alpha_0(x,y)}{f(x,y)}, \\
\alpha^i_2(x,y, z_i)&= -\frac{\alpha^i_0(x,y, z_i)}{f(x,y)}  ,
\end{aligned}
\end{equation}
which are described by the affine coordinates $[x:y:z_1:\ldots:z_{n-2}:1]$.

%%%%%%%%%%%%%%%%%%%%%%%%%%%%%%%%%%%%%%%%%%%%%%%%%
\subsection{Birational transformations of the connection 
$(\nabla_1)_{\boldsymbol{\lambda}}$}\label{2019.7.10.22.00}
%%%%%%%%%%%%%%%%%%%%%%%%%%%%%%%%%%%%%%%%%%%%%%%%%
From the connection $(\nabla_1)_{\boldsymbol{\lambda}}$ on $ f(Y) \times \mathbb{C}^2 \ra f(Y)$, 
we construct a connection on the trivial bundle $\mathbb{P}^n \times \mathbb{C}^2 \ra \mathbb{P}^n$
whose pole divisor is $D_n$.
If we extend the rational 1-forms (\ref{2018.4.27.13.15}) to
rational 1-forms on $\mathbb{P}^n$,
then 
$\alpha_0(x,y)$, 
$\alpha^i_0(x,y, z_i)$ and $\alpha_1(x,y)$ have
 poles of order $2$, $2$ and $1$ along the divisor $(t=0)$, respectively.
 On the other hand, 
 the rational 1-forms $\alpha_2(x,y)$ and $\alpha_2^i(x,y, z_i)$ have no pole along
the divisor $(t=0)$.
So we consider a birational transformation of the projective connection (\ref{2018.4.27.14.06}) 
as follows.
The $dy/y$ part of the projective connection (\ref{2018.4.27.14.06}) is 
\begin{equation*}
\begin{aligned}
& d \tilde{w}_0 - \frac{\lambda_0 (\tilde{w}_0 - x+1)(\tilde{w}_0 + x-1)}{x+1} \frac{dy}{y}  \\
&\quad + [\text{ terms whose pole divisors do not contain the divisor $(y=0)$ }] .
\end{aligned}
\end{equation*}
Then we consider the following birational map
\begin{equation}\label{2018.4.13.16.49}
\begin{aligned}
\mathbb{P}^n \times \mathbb{P}^1
&\dashrightarrow \mathbb{P}^n \times \mathbb{P}^1 \\
([x:y: z_1: \ldots: z_{n-2}:1],[1: \tilde{w}_0]) \
&\longmapsto ( [x: y: z_1: \ldots  :z_{n-2}:1], [1:w]),
\end{aligned}
\end{equation}
where $\tilde{w}_0 -x+1 = w y$.
By this birational transformation (\ref{2018.4.13.16.49}), 
the projective connection (\ref{2018.4.27.14.06}) is transformed into
\begin{equation*}
\begin{aligned}
& d w + \left( \cA_{21} (x,y) 
+ \sum^{n-2}_{i=1}  \cA_{21}^i (x,y, z_i)\right) w^2 \\
& \  + 2
 \left(\cA_{11}^i(x,y)
 + \sum_{i=1}^{n-2}   \cA_{11}^i (x,y, z_i)
 \right) 
  w  +\cA_{12}^i(x,y)
 + \sum_{i=1}^{n-2}   \cA_{12}^i (x,y, z_i),
\end{aligned}
\end{equation*}
where
\begin{equation}\label{2018.5.21.12.20}
\begin{aligned}
\cA_{21} (x,y)&:=  y  \alpha_2 (x,y), \\
\cA_{21}^i (x,y,z_i)&:= y \alpha^i_2 (x,y, z_i) , \\
\cA_{11} (x,y)&:=
 (x-1)  \alpha_2 (x,y)  
 +  \alpha_1(x,y) +\frac{1}{2}\frac{d y}{y} ,\\
\cA_{11}^i (x,y,z_i)&:=  (x-1)  \alpha^i_2 (x,y, z_i) ,\\
\cA_{12} (x,y) &:= \frac{dx    + (x-1)^2\alpha_2 (x,y)
 + 2 (x -1)\alpha_1(x,y)  + \alpha_0(x,y)  }{y} , \\
\cA_{12}^i (x,y,z_i)&:=
 \frac{(x-1)^2\alpha^i_2 (x,y, z_i) + \alpha^{i}_0(x,y, z_i)  }{y} .
\end{aligned}
\end{equation}
The corresponding connection $\nabla_{\boldsymbol{\lambda}}$ on
$ \mathbb{P}^n \times \mathbb{C}^2 \ra \mathbb{P}^n$ is 
\begin{equation*}
\begin{aligned}
\nabla_{\boldsymbol{\lambda}}
=d+ 
\begin{pmatrix}
\cA_{11} (x,y) & \cA_{12} (x,y) \\
-\cA_{21} (x,y) &- \cA_{11} (x,y)
\end{pmatrix} + \sum_{i=1}^{n-2}
\begin{pmatrix}
\cA_{11}^i (x,y,z_i) & \cA_{12}^i (x,y,z_i) \\
-\cA_{21}^i (x,y,z_i) &- \cA_{11}^i (x,y,z_i)
\end{pmatrix},
\end{aligned}
\end{equation*}
whose polar divisor is $D_n$.
This connection $\nabla_{\boldsymbol{\lambda}}$ 
is the connection described in Section \ref{2019.7.10.22.19}.
We consider a relation between $\nabla_{\boldsymbol{\lambda}}$ and 
$(\nabla_1)_{\boldsymbol{\lambda}}$.
Let $\nabla_1'$ be the meromorphic connection 
on $\mathbb{P}^n \times \mathbb{C} \ra \mathbb{P}^n$ defined by
$\nabla_1':=d-\frac{1}{2}\frac{dy}{y}$.
We define a matrix $M_2(x,y)$ on $Y$ by
\begin{equation*}
M_2(x,y):=  
\begin{pmatrix}
y & x-1\\
0 &1
\end{pmatrix} .
\end{equation*}
Let $\nabla_1''$ be the meromorphic connection 
on $\mathbb{P}^n \times \mathbb{C}^2 \ra \mathbb{P}^n$ defined by
\begin{equation*}
\begin{aligned}
\nabla_1''=d+& M_2(x,y)^{-1} dM_2(x,y)
+
M_2(x,y)^{-1}
\begin{pmatrix}
\alpha_1(x,y)  & \alpha_0(x,y)  \\
-\alpha_2(x,y)  &- \alpha_1(x,y) \\
\end{pmatrix}
M_2(x,y)\\
&\quad +
\sum_{i=1}^{n-2}
M_2(x,y)^{-1}
\begin{pmatrix}
0  & 
 \alpha^i_0(x,y, z_i) \\
-\alpha^i_2(x,y, z_i) & 0 \\
\end{pmatrix}
M_2(x,y).
\end{aligned}
\end{equation*}
We can check that 
\begin{equation}\label{2018.5.8.23.54}
 \nabla_{\boldsymbol{\lambda}} = 
\nabla''_1 \otimes \nabla'_1.
\end{equation}
By a combination of the equalities (\ref{2018.5.8.23.53}) and (\ref{2018.5.8.23.54}),
we have the following proposition:
%%%%%%%%%%%%%%%%%%%%%%%%%%%%%%%%%%%%%%%%%%%%%%%%
\begin{prop}\label{2019.7.12.12.28}
The pull-back $ f^* \nabla_{\boldsymbol{\lambda}}$ 
is birationally equivalent to 
$(\nabla_0)_{\boldsymbol{\lambda}}\otimes \nabla_0'\otimes f^* \nabla_1'$.
\end{prop}
%%%%%%%%%%%%%%%%%%%%%%%%%%%%%%%%%%%%%%%%%%%%%%%%%

\begin{proof}[Proof of Theorem \ref{2019.7.10.21.43} and Theorem \ref{2019.7.10.15.33}]
First, since $(\nabla_0)_{\boldsymbol{\lambda}}$,
$\nabla_0'$ and $\nabla_1'$ are flat and 
$f$ is a generically finite Galois morphism, we have 
the flatness of $\nabla_{\boldsymbol{\lambda}}$ for each $\boldsymbol{\lambda}$
by Proposition \ref{2019.7.12.12.28}.
Second, we have that $\nabla_{\boldsymbol{\lambda}}$
has at worst regular singularities for each $\boldsymbol{\lambda}$ 
by the explicit expression of $\nabla_{\boldsymbol{\lambda}}$ and 
\cite[Chap. II, Theorem 4.1 (ii)]{Deligne}.
Finally, the assertion of Theorem \ref{2019.7.10.15.33}
is deduced by Proposition \ref{2019.7.12.12.28}.
\end{proof}

%%%%%%%%%%%%%%%%%%%%%%%%%%%%%%%%%%%%%%%%%%%%%%%%%%
%%%%%%%%%%%%%%%%%%%%%%%%%%%%%%%%%%%%%%%%%%%%%%%%%%
\section{Monodromy representation}\label{2018.5.18.17.14}
%%%%%%%%%%%%%%%%%%%%%%%%%%%%%%%%%%%%%%%%%%%%%%%%%%
%%%%%%%%%%%%%%%%%%%%%%%%%%%%%%%%%%%%%%%%%%%%%%%%%%

In this section, we consider the monodromy representation 
$\pi_1(\mathbb{P}^n \setminus D_n,*) \ra \mathrm{SL}_{2}(\mathbb{C})$
of $\nabla_{\boldsymbol{\lambda}}$ for generic $\boldsymbol{\lambda}$.
In Section \ref{2019.7.10.23.32} and Section \ref{2019.7.10.23.33},
we discuss structure of the fundamental group $\pi_1(\mathbb{P}^n \setminus D_n,*)$
by using the Zariski's hyperplane section theorem 
and the Zariski--Van-Kampen method.
In Section \ref{2019.7.10.22.02}, we show Theorem \ref{2019.7.10.22.03} 
by using the results in Section \ref{2019.7.10.23.32} and Section \ref{2019.7.10.23.33}.

%%%%%%%%%%%%%%%%%%%%%%%%%%%%%%%%%%%%%%%%%%%%%%%%%
\subsection{Zariski's hyperplane section theorem}\label{2019.7.10.23.32}
%%%%%%%%%%%%%%%%%%%%%%%%%%%%%%%%%%%%%%%%%%%%%%%%%

Let $H_i$ ($i=1,\ldots,n-2$) be the hyperplanes in $\mathbb{P}^n$ defined by
\begin{equation*}
H_i := ( z_i - a_i x - b_i y -c_i t =0 ) \quad i=1,\ldots , n-2.
\end{equation*}
Here $a_i, b_i$, and $c_i $ $(i=1,\dots,n-2)$ are generic complex numbers.
For simplicity, we assume that $0<|a_i|\ll 1$ and $ 0 < |b_i | \ll 1$.
Let $f(x,y,t)$ be the following quadratic polynomial
\begin{equation*}
\begin{aligned}
f(x,y,t) :=&\ x^2 + y^2 +t^2  -2 ( x y +y t +  t x) \\
=&\ (y-x-t)^2 - 4 xt  .
\end{aligned}
\end{equation*}
Let $\tilde{\mathcal{Q}}_0$, $\tilde{\mathcal{Q}}_i$ and $\tilde{D}_n$ 
be the divisors on $\mathbb{P}^2=\mathbb{P}^n \cap (\cap_{i=1}^{n-2}H_i)$ defined by
\begin{equation*}
\begin{aligned}
\tilde{\mathcal{Q}}_0&:= (f(x,y,t)=0 ),   \\
\tilde{\mathcal{Q}}_i&:= (f(x,y,t) - ( a_i x + b_i y +c_i t)^2 =0)
\quad  (i=1,\ldots,n-2), \text{ and } \\
\tilde{D}_n &:= (x=0 ) + (y=0) + ( t=0) + 
\tilde{\mathcal{Q}}_0 + \tilde{\mathcal{Q}}_1 + \cdots + \tilde{\mathcal{Q}}_{n-2},
\end{aligned}
\end{equation*}
respectively.
By Zariski's hyperplane section theorem (for example see \cite{HL}), we have the natural isomorphism
\begin{equation*}
\pi_1(\mathbb{P}^n \setminus D_n,*) \cong \pi_1(\mathbb{P}^2 \setminus \tilde{D}_n,*).
\end{equation*}

%%%%%%%%%%%%%%%%%%%%%%%%%%%%%%%%%%%%%%%%%%%%%%%%
\subsection{Zariski--Van-Kampen method}\label{2019.7.10.23.33}
%%%%%%%%%%%%%%%%%%%%%%%%%%%%%%%%%%%%%%%%%%%%%%%%

We derive some equalities in $\pi_1(\mathbb{P}^2 \setminus \tilde{D}_n,*)$
 by the Zariski--Van-Kampen method (see for example \cite{Degtyarev}).

Let $\pi\colon \mathbb{P}^{2} \setminus \tilde{D}_n \ra \mathbb{P}^1$ be 
the projection defined by
\begin{equation*}
\begin{aligned}
\pi\colon \mathbb{P}^{2} \setminus \tilde{D}_n &\lra \mathbb{P}^1 \\
[x:y:t]&\longmapsto [x:t].
\end{aligned}
\end{equation*}
Let $\{[x^+_i:1] ,[ x_i^-:1 ]\} \subset \mathbb{P}^1$ 
be the roots of the discriminant of $f(x,y,t) - ( a_i x + b_i y +c_i t)^2$ with respect to $y$.
We denote $[x^+_i:1]$ and $[x^-_i:1]$ by $x^+_i$ and $x^-_i$, respectively.
Since $0<|a_i|\ll 1$ and $ 0 < |b_i | \ll 1$, 
there exists an element of $\{x^+_i , x_i^-\}$ in a neighborhood of $\infty=[0:1]$.
We assume that $x_i^-$ is a point in a neighborhood of $\infty$.
Set $a=[ a:1 ]$ where $0<|a|\ll 0$.
For $i=0,1,\dots,n-2$,
let $y^+_{i}$ and $y^-_{i}$ be the intersection of 
$\tilde{\mathcal{Q}}_i$ and $\pi^{-1}(a)$:
$\tilde{\mathcal{Q}}_i\cap \pi^{-1}(a) = \{ y^+_{i} , y^-_{i} \} $.
Here we assume that 
\begin{equation*}
0<\mathrm{Arg} \left( \frac{y^+_1-(a+1)}{y^+_0-(a+1)} \right) 
<\cdots<\mathrm{Arg} \left( \frac{y^+_{n-2}-(a+1)}{y^+_0-(a+1)} \right)
< \pi.
\end{equation*}
\begin{figure}[h]
\caption{Fibers of $\pi$.}
\begin{center}
\includegraphics[clip,width=9cm]{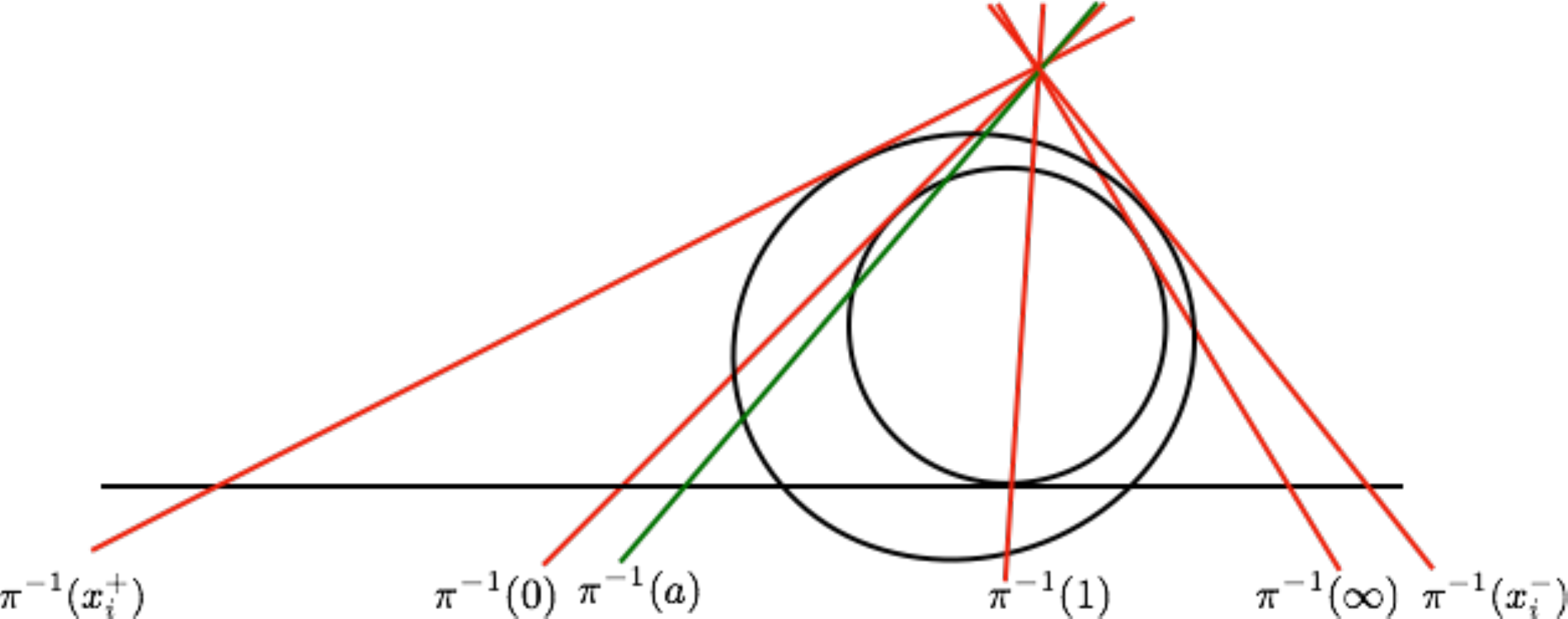}
\end{center}
\end{figure}

We define natural numbers $i_k$ and $j_k$ ($k=1,\dots,n-2$) 
so that 
$\{i_1 ,\ldots,i_{n-2} \}=\{1 ,\ldots,n-2 \}$,
$\{j_1 ,\ldots,j_{n-2} \}=\{1 ,\ldots,n-2 \}$, 
\begin{equation*}
0<\mathrm{Arg}\,  (x^+_{i_1})
<\cdots<\mathrm{Arg}\,  (x^+_{i_{n-2}}) 
< 2\pi \text{ and } 
0<\mathrm{Arg}\,  \frac{1}{x^-_{j_1}} 
<\cdots<\mathrm{Arg}\,  \frac{1}{x^-_{j_{n-2}}} 
< 2\pi.
\end{equation*}
Here we define the range of the principal value of arguments $\mathrm{Arg}$
by the closed-open interval $[0,2\pi)$.
Let $\Gamma$ 
%and $\Gamma_2$ 
be the group defined by
\begin{equation*}
\begin{aligned}
\Gamma &:=
\left\langle 
\begin{array}{l}
\alpha_0, \alpha_{y^+_0} ,\ldots , \alpha_{y^+_{n-2}} \\
\alpha_{y^-_0} ,\ldots , \alpha_{y^-_{n-2}},\alpha_{\infty} 
\end{array}
\middle| \ 
\alpha_0\alpha_{y^+_{1}}    \cdots \alpha_{y^+_{n-2}} \alpha_{y^-_{0}}
\alpha_{y^-_{1}}    \ldots \alpha_{y^-_{n-2}}
\alpha_{y^+_0} 
 \alpha_{\infty}=1 \right\rangle .
 %\text{ and} \\
%\Gamma_2
%&:= \left\langle 
%\begin{array}{l}
%\end{array}
% \middle| 
%\begin{array}{l}
%\gamma_{x^-_{j_1}} \cdots \gamma_{x^-_{j_{n-2}}}
 %\gamma_{\infty} \gamma_1  \gamma_{x^+_{i_1}} \cdots \gamma_{x^+_{i_{n-2}}}   \gamma_0=1
%\end{array}
%\right\rangle.
\end{aligned}
\end{equation*}
Then we have 
$\pi_1 (\pi^{-1} (a)  \setminus (\tilde{D}_n\cap \pi^{-1}(a) ) , *  ) \cong \Gamma$ 
and have an exact sequence
\begin{equation*}
 1  \longrightarrow \Gamma 
 \longrightarrow \pi_1(\mathbb{P}^2 \setminus \tilde{D}_n,*) 
 \longrightarrow \pi_1(\mathbb{P}^1 \setminus \{ 0,\infty \},a) \longrightarrow 1.
\end{equation*}
Let
\begin{equation}\label{1.21.16.34}
\gamma_0, \gamma_1 ,  \gamma_{x^+_1}, \ldots , \gamma_{x^+_{n-2}} ,
\gamma_{x^-_1}, \ldots , \gamma_{x^-_{n-2}}, \gamma_{\infty}
\end{equation}
be loops with base point $a$
on $\mathbb{P}^1\setminus \{ 0, 1 , x^{\pm}_1 ,\ldots ,x^{\pm}_{n-1}, \infty \}$ 
such that for $x\in \{ 0, 1 , x^{\pm}_1 ,\ldots ,x^{\pm}_{n-1}, \infty \}$,
the loop
$\gamma_x$ is oriented counter-clockwise, 
$x$ lies inside, while the other points 
$\{ 0, 1 , x^{\pm}_1 ,\ldots ,x^{\pm}_{n-1}, \infty \}\setminus \{x\}$ 
lie outside as in Figure 2.
\begin{figure}[h]
\caption{Loops on $\pi^{-1}(a)$ and $\mathbb{P}^1=(\mathbb{C}_x)_0 \cup (\mathbb{C}_x)_{\infty}.$}
\begin{center}
\includegraphics[clip,width=15cm]{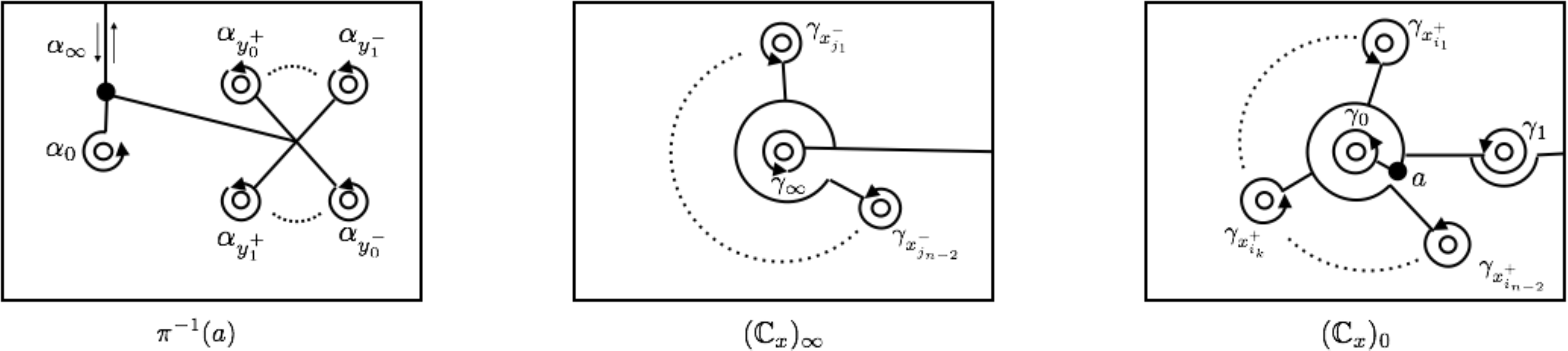}
\end{center}
\end{figure}

Let $s \colon \mathbb{P}^1 \setminus \{ 0,\infty \} \rightarrow \mathbb{P}^2 \setminus \tilde{D}_n $
be a continuous section of $\pi$ such that
$s(a)=* \in \mathbb{P}^2 \setminus \tilde{D}_n$.
For the loops (\ref{1.21.16.34}), we define the monodromy actions of the
loops (\ref{1.21.16.34}) on $\Gamma$ as in \cite[Theorem 2.2.1]{Cousin}.
Namely, the action $(\gamma_x,\alpha) \mapsto \gamma_x(\alpha)$
for loops $\gamma_x$ and $\alpha \in \Gamma$  is characterized by the equality
$\gamma_x(\alpha) = \gamma_x^{-1} \alpha \gamma_x$ 
in $\pi_1(\mathbb{P}^2 \setminus \tilde{D}_n,*) $.
Here we denote by $\gamma_x$ 
the loop $s_*(\gamma_x) \in \pi_1(\mathbb{P}^2 \setminus \tilde{D}_n,*) $
for simplicity.
For explicit computation of this action, we consider the motion of the points
$y^{\pm}_i$ ($i=0,1,\ldots,n-2$) when $a$ varies along the loop $\gamma_x$
and a continuous deformation of 
$\alpha\in \pi_1 (\pi^{-1} (a)  \setminus (\tilde{D}_n\cap \pi^{-1}(a) ) , *  )$
according to the motion of these points.
Note that the assumptions
$0<|a_i|\ll 1$, $ 0 < |b_i | \ll 1$ and $0<|a| \ll 1$ 
make the computation of the motion of the points $y_i^{\pm}$ simple.
By explicit computation of the action on some loops, we can check the following equalities:
\begin{align}
\gamma_0(\alpha_{y^+_0}) &= \alpha_{y^-_{0}}; &
\gamma_{x^+_i}(\alpha_{y^+_i}) &= \alpha_{y^+_i} \alpha_{y^-_{i}} \alpha_{y^+_i}^{-1} 
\ (i=1,\ldots , n-2); \label{2018.5.12.13.30}\\
\gamma_0(\alpha_{0}) &= \alpha_{0};   &
\gamma_0(\alpha_{y^+_i})& = \alpha_{y^+_{i}}   
\ (i=1,\dots,n-2) ; \label{2019.1.21.16.48} \\
\tilde{\gamma}_{\infty}(\alpha_{y^+_0}) &=  \alpha_{0}
\alpha_{y^-_0} \alpha_{0}^{-1}; \label{2018.5.12.13.28} 
\end{align}
and
\begin{equation}\label{2019.1.12.16.51}
\begin{aligned}
\gamma_1(\alpha_{y^+_0}) &= ( \alpha_0 \alpha_{y^+_0}) \alpha_0   
\alpha_{y^+_0}\alpha_0^{-1} (\alpha_{0} \alpha_{y^+_0})^{-1}  .
\end{aligned}
\end{equation}
Here we put
 $\tilde{\gamma}_{\infty}:= \gamma_{x^-_{j_1}} \cdots \gamma_{x^-_{j_{n-2}}} \gamma_{\infty}$.
In fact, if $a$ varies along the loop $\gamma_0$,
then $y_0^+$ moves to the location of $y_0^-$,
and $0$ and $y_i^+$ ($i=1,\ldots,n-2$) go back to the
prior locations, respectively.
If $a$ varies along the loop $\gamma_{x^{+}_i}$ for $i=1,\ldots,n-2$,
then $y_i^+$ moves to the location of $y_i^-$.
Here we assume that $-1 \ll \mathrm{Arg}\, (a)<0$
and $y_0^+$ closes to $0$ when $a$ approach $1$ along the real axis.
If $a$ varies along the loop $\tilde{\gamma}_{\infty}$,
then $y_0^+$ moves to the location of $y_0^-$ round by $0$.
If $a$ varies along the loop $\gamma_1$, 
then $y_0^+$ go back to the prior locations round by $0$ twice. 
If we consider continuous deformations of the corresponding loops 
according to the motion of these points, 
then we have the equalities (\ref{2018.5.12.13.30}), (\ref{2019.1.21.16.48}),
(\ref{2018.5.12.13.28}) and (\ref{2019.1.12.16.51}).
Here note that the images 
of the intersection of $\mathcal{Q}_i$ and $\mathcal{Q}_j$ under $\pi$ 
are close to $\infty$ for $i,j =0,1, \ldots, n-2$
since $0<|a_i|\ll 1$ and $ 0 < |b_i | \ll 1$.

%We take a lift of $\gamma_0$ such that 
%$\gamma_0 (\alpha_y ) = \gamma_0^{-1} \alpha_y \gamma_0$ for 
%$y \in \tilde{D}_n\cap \pi^{-1}(a)$. 
%Here this lift is an element of $\pi_1(\mathbb{P}^2 \setminus \tilde{D}_n,*)$.
%We denote by the same notation $\gamma_0$ this lift.   
In $\pi_1(\mathbb{P}^2 \setminus \tilde{D}_n,*) $,
we have the equality
$\gamma_x(\alpha) = \gamma_x^{-1} \alpha \gamma_x$ for $\alpha \in \Gamma$.
By the equality (\ref{2018.5.12.13.30}), we can show that
 $\alpha_{y_i^-}$ ($i=0,\ldots,n-2$) are generated by 
$\alpha_{y^+_0} ,\ldots , \alpha_{y^+_{n-2}}$, and $\gamma_0$
in $\pi_1(\mathbb{P}^2 \setminus \tilde{D}_n,*)$.
Then we obtain the following proposition.

%%%%%%%%%%%%%%%%%%%%%%%%%%%%%%%%%%%%%%%%%%%%%%%%%%
\begin{prop}
The group $\pi_1(\mathbb{P}^2 \setminus \tilde{D}_n,*)$ is generated by
$\alpha_0, \alpha_{y^+_0} ,\ldots , \alpha_{y^+_{n-2}}$, and
$\gamma_0$.
\end{prop}
%%%%%%%%%%%%%%%%%%%%%%%%%%%%%%%%%%%%%%%%%%%%%%%%%%

%%%%%%%%%%%%%%%%%%%%%%%%%%%%%%%%%%%%%%%%%%%%%%%%%%
\begin{prop}
Set $\tilde{\alpha}= \alpha_{y^+_1} \cdots \alpha_{y^+_{n-2}}$.
For the elements $\alpha_{y^+_0}, \alpha_0,\gamma_{0}$, and $\tilde{\alpha}$ 
of $\pi_1(\mathbb{P}^2 \setminus \tilde{D}_n,*)$,
we have the following equalities:
\begin{align}
\ [ \alpha_0 , \gamma_{0} ]&= [ \tilde{\alpha}, \gamma_0 ]=1 \label{2018.5.10.22.52}, \\
((\gamma_{0}\tilde{\alpha} ) \alpha_{y^+_0})^2&=
(\alpha_{y^+_0}(\gamma_{0}
 \tilde{\alpha} ))^2, \label{2018.5.10.22.54} \\
(\alpha_{y^+_0} \alpha_0)^2 &= ( \alpha_0 \alpha_{y^+_0})^2 \label{2018.5.12.13.38}.
\end{align}
\end{prop}
%%%%%%%%%%%%%%%%%%%%%%%%%%%%%%%%%%%%%%%%%%%%%%%%%%

\begin{proof}\label{2018.5.12.16.43}
By the equality (\ref{2019.1.21.16.48}),
we have the equality (\ref{2018.5.10.22.52}).
Second, we show the equality (\ref{2018.5.10.22.54}).
By the equalities
(\ref{2018.5.12.13.30}), (\ref{2018.5.12.13.28}), 
$a_{\infty}=\gamma_{\infty}\gamma_0$, and (\ref{2018.5.10.22.52}),
we have
\begin{equation*}
\begin{aligned}
a_{y^+_0}&=  \gamma_{\infty} \alpha_{0}
\gamma_0^{-1}\alpha_{y^+_0}\gamma_0
\alpha_{0}^{-1} \gamma_{\infty}^{-1} \\
%&= \alpha_{\infty} \gamma_{0}^{-1} \alpha_{0}
%\gamma_0^{-1}\alpha_{y^+_0}\gamma_0
%\alpha_{0}^{-1}  \gamma_{0}\alpha_{\infty}^{-1}\\
&= \alpha_{\infty} \alpha_{0}
\gamma^{-1}_{0}\gamma_0^{-1}\alpha_{y^+_0}\gamma_0\gamma_{0}
\alpha_{0}^{-1}\alpha_{\infty}^{-1} \\
&=
(\alpha_{y^+_1} \cdots \alpha_{y^+_{n-2}} \alpha_{y^-_0}
 \alpha_{y^-_1} \cdots \alpha_{y^-_{n-2}}  \alpha_{y^+_0})^{-1}
\gamma_{0}^{-1} \gamma_0^{-1}\alpha_{y^+_0}\gamma_0\gamma_0
(\alpha_{y^+_1} \cdots \alpha_{y^+_{n-2}} \alpha_{y^-_0}
 \alpha_{y^-_1} \cdots \alpha_{y^-_{n-2}} \alpha_{y^+_0})
 \\
 &=
( \tilde{\alpha} \gamma_{0}^{-1} \alpha_{y^+_0} \gamma_0
 \tilde{\alpha} \alpha_{y^+_0})^{-1}
\gamma_{0}^{-1}\gamma_0^{-1}\alpha_{y^+_0}\gamma_0\gamma_{0}
( \tilde{\alpha} \gamma_{0}^{-1} \alpha_{y^+_0} \gamma_0
 \tilde{\alpha} \alpha_{y^+_0})\\
  &=
( \tilde{\alpha}  \alpha_{y^+_0} \gamma_0
 \tilde{\alpha} \alpha_{y^+_0})^{-1}
\gamma_{0}^{-1}\alpha_{y^+_0}\gamma_0
( \tilde{\alpha}  \alpha_{y^+_0} \gamma_0
 \tilde{\alpha} \alpha_{y^+_0}).
\end{aligned}
\end{equation*}
Then we have the equality (\ref{2018.5.10.22.54}).
By the equality (\ref{2019.1.12.16.51}), we have the equality (\ref{2018.5.12.13.38}).
\end{proof}

%%%%%%%%%%%%%%%%%%%%%%%%%%%%%%%%%%%%%%%%%%%%%%%
\subsection{Monodromy representation of $\nabla_{\boldsymbol{\lambda}}$}\label{2019.7.10.22.02}
%%%%%%%%%%%%%%%%%%%%%%%%%%%%%%%%%%%%%%%%%%%%%%%

Let $\boldsymbol{D}_{\infty}$ be the infinite dihedral group:
\begin{equation*}
\boldsymbol{D}_{\infty}:= \left\langle
\begin{pmatrix}
0 & \alpha \\
-\alpha^{-1} & 0
\end{pmatrix},
\begin{pmatrix}
\beta & 0 \\
0 & \beta^{-1}
\end{pmatrix}\ 
\middle| \ 
\alpha , \beta \in \mathbb{C}^*
 \right\rangle \le \mathrm{SL}_2 (\mathbb{C}).
\end{equation*}

%%%%%%%%%%%%%%%%%%%%%%%%%%%%%%%%%%%%%%%%%%%%%%%%%%
\begin{prop}\label{2018.5.8.15.13}
For generic $\boldsymbol{\lambda}$, the monodromy representation of 
$\nabla_{\boldsymbol{\lambda}}$ is conjugated to 
the dihedral representation $\rho_{\boldsymbol{\lambda}} \colon 
\pi_1(\mathbb{P}^n \setminus D_n,*) \ra \boldsymbol{D}_{\infty}$ of 
the fundamental group $\pi_1(\mathbb{P}^n \setminus D_n,*) $ 
defined by
\begin{equation*}
\begin{aligned}
\rho_{\boldsymbol{\lambda}}(\alpha_0)&= 
\begin{pmatrix}
-\exp (- \pi \lambda_{0} ) & 0 \\
0 & -\exp ( \pi \lambda_{0} )
\end{pmatrix}, &
\rho_{\boldsymbol{\lambda}}(\gamma_0)&=\begin{pmatrix}
\exp (- \pi \lambda_1 ) & 0 \\
0 & \exp ( \pi \lambda_1 )
\end{pmatrix}, \\
\rho_{\boldsymbol{\lambda}}(\alpha_{y^+_0})&=
\begin{pmatrix}
0 & 1 \\
-1 & 0
\end{pmatrix}, &
\rho_{\boldsymbol{\lambda}}(\alpha_{y^+_i})&=
\begin{pmatrix}
\exp (- \pi \lambda_{i+1} ) & 0 \\
0 & \exp ( \pi \lambda_{i+1} )
\end{pmatrix}, 
\end{aligned}
\end{equation*}
where $i=1,\ldots , n-2$.
\end{prop}
%%%%%%%%%%%%%%%%%%%%%%%%%%%%%%%%%%%%%%%%%%%%%%%%%%

\begin{proof}
Let $\rho_{\nabla_{\boldsymbol{\lambda}}}\colon 
\pi_1(\mathbb{P}^n \setminus D_n,*) \ra \mathrm{SL}_{2}(\mathbb{C})$ be
a monodromy representation of $\nabla_{\boldsymbol{\lambda}}$.
Put $A_0 := \rho_{\nabla_{\boldsymbol{\lambda}}}(\alpha_0)$, 
$A_{y^+_i} := \rho_{\nabla_{\boldsymbol{\lambda}}}(\alpha_{y^+_i})$ 
($i=0,\ldots,n-1$) and $C_0 := \rho_{\nabla_{\boldsymbol{\lambda}}}(\gamma_0)$.
Let $U$ be some analytic open subset of 
$\mathbb{P}^2 \setminus (\tilde{\cQ}_0 \cup (y=0)\cup (t=0))$
such that $U$ is simply connected and $U$ contains the loops $\alpha_{y_i^+}$ ($i=1,\ldots,n-1$)
and $\gamma_0$. 
On the open subset $U$, the connection $\nabla_{\boldsymbol{\lambda}}$ 
is isomorphic to $(\nabla_0)_{\boldsymbol{\lambda}}$. 
Then by some conjugation, we may put
\begin{equation*}
C_0=
\begin{pmatrix}
\exp (- \pi \lambda_1 ) & 0 \\
0 & \exp ( \pi \lambda_1 )
\end{pmatrix}, \quad 
A_{y^+_i}=
\begin{pmatrix}
\exp (- \pi \lambda_{i+1} ) & 0 \\
0 & \exp ( \pi \lambda_{i+1} )
\end{pmatrix}\  i=1,\ldots , n-2.
\end{equation*}
Assume that $\exp (- \pi \lambda_1 ) \neq \exp ( \pi \lambda_1 )$.
By Proposition \ref{2018.5.12.16.43}, we have the equality $A_0 C_0= C_0A_0$.
Then we have
\begin{equation*}
A_{0}=
\begin{pmatrix}
-\exp (- \pi \lambda_{0} ) & 0 \\
0 & -\exp ( \pi \lambda_{0} )
\end{pmatrix}.
\end{equation*}

Note that the image $\mathrm{Im}(\rho_{\nabla_{\boldsymbol{\lambda}}})$ is non-abelian.
Since $C_0$, $A_0$, and $A_{y_i^+}$ ($i=1,\ldots, n-2$) are diagonal matrices,
we may put
\begin{equation*}
A_{y^+_0}=
\begin{pmatrix}
a_{11} & a_{12} \\
-1 & a_{22}
\end{pmatrix}.
\end{equation*}
Put $\tilde{A}:=A_{y^+_1} \cdots A_{y^+_{n-2}}$.
By Proposition \ref{2018.5.12.16.43}, we have the equalities 
$(A_{y^+_0} (C_0 \tilde{A}) )^2=
((C_0 \tilde{A} )A_{y^+_0})^{2}$
and
$(A_{y^+_0} A_0 )^2=(A_0A_{y^+_0})^{2}$.
Assume that $(\exp (-\pi \lambda_1 - \pi \sum_{i=1}^{n-2} \lambda_{i+1}  ) )^2 \neq 1$ 
and $ (- \exp (-\pi \lambda_0 ) )^2 \neq 1 $.
Since $A_{y^+_0} (C_0 \tilde{A})\neq  (C_0 \tilde{A})A_{y^+_0}$ 
and $A_{y^+_0} A_0 \neq  A_0A_{y^+_0}$, we have the equalities 
$(A_{y^+_0}(C_0 \tilde{A} ))^{2}= - I_2$ 
and $(A_{y^+_0}A_0)^{2} = - I_2$.
Then we have the following equalities:
\begin{equation*}
\begin{cases}
a_{11}a_{22}+a_{12}=1 \\
a_{11} (\exp (-\pi \lambda_1 - \pi \sum_{i=1}^{n-2} \lambda_{i+1}  ) )^2 = a_{22} \\
a_{11} (- \exp (-\pi \lambda_0 ) )^2 = a_{22}.
\end{cases}
\end{equation*}
We assume that 
$(\exp (-\pi \lambda_1 - \pi \sum_{i=1}^{n-2} \lambda_{i+1}  ) )^2 \neq (- \exp (-\pi \lambda_0 ) )^2$.
Then we have $a_{11}=0$, $a_{12}=1$, and $a_{22}=0$.
\end{proof}

\begin{proof}[Proof of Theorem \ref{2019.7.10.22.03}]
By Theorem \ref{2019.7.10.15.33},
we have that the monodromy representation of 
$\nabla_{\boldsymbol{\lambda}}$ 
is  virtually abelian.
By Proposition \ref{2018.5.8.15.13},
we have that the monodromy representation of 
$\nabla_{\boldsymbol{\lambda}}$ 
is conjugated to the explicit
representation 
\begin{equation*}
\rho_{\boldsymbol{\lambda}} \colon \pi_1 (\mathbb{P}^n \setminus D_n) 
\longrightarrow 
\mathrm{SL}_2(\mathbb{C}),
\end{equation*}
which takes values in the infinite dihedral group $\boldsymbol{D}_{\infty}$.
\end{proof}

%%%%%%%%%%%%%%%%%%%%%%%%%%%%%%%%%%%%%%%%%%%%%%%%%%
%%%%%%%%%%%%%%%%%%%%%%%%%%%%%%%%%%%%%%%%%%%%%%%%%%%
\section{Algebraic Garnier solution}\label{2018.5.18.17.17}
%%%%%%%%%%%%%%%%%%%%%%%%%%%%%%%%%%%%%%%%%%%%%%%%%%
%%%%%%%%%%%%%%%%%%%%%%%%%%%%%%%%%%%%%%%%%%%%%%%%%%%

Assume that an $n$-tuple of
complex numbers $\boldsymbol{\lambda}=(\lambda_0, \ldots, \lambda_{n-1} )$
is sufficiently generic.
In this section, we restrict the flat connection $\nabla_{\boldsymbol{\lambda}}$ to 
a generic line $\mathbb{P}^n\cap (\cap_{i=0}^{n-2} H'_i)$, where
\begin{equation}\label{2018.5.3.19.42}
\begin{cases}
H'_0 = (y-ax -b t=0) &\\
H'_i =(z_i - c_{i} x - d_i t=0) & (i=1,2,\dots,n-2).
\end{cases}
\end{equation}
Here $a, b, c_i $, and $d_i$ $(i=1,2,\dots,n-2)$ are generic complex numbers.
We consider the transformation $\tilde{x}= -\frac{a}{b}x$.
Let $T$ be a Zariski open subset of $\Spec \mathbb{C} [a,b, c_i ,d_i ]_{i=1,\ldots,n-2}$.
We consider the map $\mathbb{P}^1 \times T \ra \mathbb{P}^n$ defined by (\ref{2018.5.3.19.42}).
Let $(\nabla_{\mathbb{P}^1 \times T})_{\boldsymbol{\lambda}}$ be the flat connection 
on the trivial rank $2$ vector bundle $F_0$ over $\mathbb{P}^1 \times T$ 
induced by the flat connection $\nabla_{\boldsymbol{\lambda}}$ over $\mathbb{P}^n$.
Let 
\begin{equation*}
(\nabla_{\mathbb{P}^1 \times T/T})_{\boldsymbol{\lambda}}
\colon F_0 \lra F_0 \otimes \Omega^1_{\mathbb{P}^1 \times T/T}(D_n)
\end{equation*}
 be the relative connection 
on $F_0$ over $\mathbb{P}^1 \times T$ 
associated to $(\nabla_{\mathbb{P}^1 \times T})_{\boldsymbol{\lambda}}$.
In Section \ref{2019.7.10.22.11}, 
we introduce an \'etale base change $\tilde{T} \rightarrow T$
to prove the assertion (i) of Theorem \ref{2019.7.10.22.13}.
In Section \ref{2019.7.12.11.27},  
after the \'etale base change $\tilde{T} \rightarrow T$,
we compute the residue matrix of 
$(\nabla_{\mathbb{P}^1 \times \tilde{T}/\tilde{T}})_{\boldsymbol{\lambda}}$ 
for each simple pole.
In Section \ref{2019.7.10.22.15}, 
we recall the relation between isomonodromic deformations and the Garnier system
following \cite{Mazz}.
In Section \ref{2019.7.12.11.28},
we show Theorem \ref{2019.7.10.22.13}.

%%%%%%%%%%%%%%%%%%%%%%%%%%%%%%%%%%%%%%%%%%%%%%%%%%
\subsection{Regular singular points of 
$(\nabla_{\mathbb{P}^1 \times T/T})_{\boldsymbol{\lambda}}$}\label{2019.7.10.22.11}
%%%%%%%%%%%%%%%%%%%%%%%%%%%%%%%%%%%%%%%%%%%%%%%%%%

By the pull-back of $x^2 + y^2 +t^2  -2 (x y + y t +t x )   $ 
and $x^2 + y^2 +t^2  -2 (x y + y t +t x ) -z_i^2  $
under $\mathbb{P}^1 \times T \ra \mathbb{P}^n$, we have the following polynomials over $T$:
\begin{equation*}
\begin{aligned}
f(a,b,\tilde{x}) &:=
\frac{(a-1)^2b^2}{a^2} \tilde{x}^2 + 
\frac{2b(1+a+b-ab)}{a} \tilde{x} + (b-1)^2, \\
f_i(a,b,c_i,d_i,\tilde{x}) &:=f(a,b,\tilde{x})
-\frac{(a d_i -b c_i \tilde{x})^2}{a^2},
\end{aligned}
\end{equation*}
which are described on the affine coordinate $[\tilde{x}:1]$.
Let $I$ be the ideal of $\mathbb{C} [a,b,\tilde{b}, c_i ,d_i , \tilde{d}_i]_{i=1,\ldots,n-2}$ defined by
$I:=
(\tilde{b}^2-4 (a+b-ab), 
\tilde{d}_i^2-\Delta^i_{\tilde{x}} )_{i=1,\ldots,n-2}$,
where $\Delta^i_{\tilde{x}} $ is the discriminant of $f_i(a,b,c_i,d_i,\tilde{x})$ with respect to $\tilde{x}$.
We have the natural morphism
\begin{equation*}
\Spec \mathbb{C} [a,b,\tilde{b}, c_i ,d_i , \tilde{d}_i]_{i=1,\ldots,n-2}
/ I \lra \Spec \mathbb{C} [a,b, c_i ,d_i ]_{i=1,\ldots,n-2}.
\end{equation*}
Let $\tilde{T}$ be the inverse image of $T$ under this morphism: $\tilde{T} \ra T$.
Let $t_{1}$ and $t_{2}$ be the rational functions on $\tilde{T}$ defined by
\begin{equation*}
t_{1}:=
\frac{a( \tilde{b}-2 )^2}{4(a-1)(\tilde{b}^2-a)} \text{  and  }
t_{2}:=
\frac{a(\tilde{b}+2 )^2}{4(a-1)(\tilde{b}^2-a)}.
\end{equation*}
Then $f(a,b,t_1)=f(a,b,t_2)=0$.
Moreover,
let $t_{2i+1}$ and $t_{2i+2}$ be the rational functions on $\tilde{T}$ defined by
\begin{equation*}
t_{2i+1}:=
\frac{2ab(1+a+b-ab- c_i d_i) -a^2  \tilde{d}_i}{2b^2((a-1)^2-c_i^2)} \text{  and  }
t_{2i+2}:=
\frac{2ab(1+a+b-ab- c_i d_i) +a^2  \tilde{d}_i}{2b^2((a-1)^2-c_i^2)}.
\end{equation*}
Then $ f_i (a,b,c_i,d_i,t_{2i+1})= f_i (a,b,c_i,d_i,t_{2i+2})=0$.
By these rational functions, we have a generically finite morphism
\begin{equation}\label{2018.5.12.23.51}
\tilde{T} \lra \Spec \mathbb{C}[t_1,t_2,\ldots ,t_{2n-2}],
\end{equation}
if the Zariski open subset $T$ shrinks.
We take the pull-back $(\nabla_{\mathbb{P}^1 \times \tilde{T}/\tilde{T}})_{\boldsymbol{\lambda}}$ 
of $(\nabla_{\mathbb{P}^1 \times T/T})_{\boldsymbol{\lambda}}$ 
under the morphism $\mathbb{P}^1 \times \tilde{T} \ra \mathbb{P}^1 \times T$.
Then $(\nabla_{\mathbb{P}^1 \times \tilde{T}/\tilde{T}})_{\boldsymbol{\lambda}}$ 
is a family of the
 Fuchsian systems with $2n + 1$ regular singularities 
 at $\tilde{x}=0,1,t_1,\ldots , t_{2n-2}, \infty$ parametrized by $\tilde{T}$.

%%%%%%%%%%%%%%%%%%%%%%%%%%%%%%%%%%%%%%%%%%
\subsection{Residue matrices of 
$(\nabla_{\mathbb{P}^1 \times \tilde{T}/\tilde{T}})_{\boldsymbol{\lambda}}$ }\label{2019.7.12.11.27}
%%%%%%%%%%%%%%%%%%%%%%%%%%%%%%%%%%%%%%%%%%

We describe the residue matrices of 
$(\nabla_{\mathbb{P}^1 \times \tilde{T}/\tilde{T}})_{\boldsymbol{\lambda}}$
at the regular singular points.
Put 
\begin{equation*}
\begin{aligned}
M_2(\tilde{x})&:=
\begin{pmatrix}
-ab (\tilde{x} -1) & - b \tilde{x} -a\\
0 &a
\end{pmatrix} \text{ and}\\
\alpha_0^i(\tilde{x})&:= \lambda_{i+1} \left(
- \frac{ b c_i}{a} -
\frac{ a d_i -b c_i  \tilde{x} }{2a(\tilde{x} -t_{2i+1})} 
-\frac{ a d_i -b c_i \tilde{x} }{2a(\tilde{x} -t_{2i+2} ) } \right).
\end{aligned}
\end{equation*}
Let $H^{\tilde{T}}_{2n-1}$ be
the residue matrix at $\tilde{x}=0$.
We have the following equality
\begin{equation*}
\begin{aligned}
H^{\tilde{T}}_{2n-1}&=M_2(0)^{-1}
\begin{pmatrix}
0 & \frac{\lambda_1(\tilde{b}^2 -4)}{8(a-1)} \\
 \frac{2\lambda_1(a-1)}{\tilde{b}^2 -4} & 0
\end{pmatrix}
M_2(0).
\end{aligned}
\end{equation*}
Let $H^{\tilde{T}}_{2n}$ be
the residue matrix at $\tilde{x}=1$.
We have the following equality
\begin{equation*}
\begin{aligned}
H_{2n}^{\tilde{T}}
&=
\frac{1-\lambda_0}{2}
\begin{pmatrix}
1 & \frac{2 (\tilde{b}^2+4a^2-8a)}{\tilde{b}^2-4a^2} \\
0 & -1
\end{pmatrix}
+ 
\sum^{n-2}_{i=1} 
 \alpha_0^i(1)
\begin{pmatrix}
0 &1  +  \frac{  ( a+b)^2}{b^2(a-1)^2(1 -t_1)(1 -t_2)}\\
0 & 0
\end{pmatrix}.
\end{aligned}
\end{equation*}
Let $H^{\tilde{T}}_{1}$ and $H^{\tilde{T}}_{2}$ be
the residue matrices  at $\tilde{x}=t_{1}$ 
and $\tilde{x}=t_{2}$, respectively.
We have the following equalities
\begin{equation*}
\begin{aligned}
H^{\tilde{T}}_{1}
&=M_2(t_{1})^{-1}
 \begin{pmatrix}
-\frac{1}{4} & 0 \\
-\frac{\lambda_0(a-1)}{2(\tilde{b} -2 a)}
-\frac{\lambda_1(a-1)}{2(\tilde{b}-2)}
+ \sum^{n-2}_{i=1} 
\frac{ a^2 \alpha^i_0(t_1)}{b^2(a-1)^2(t_1 -t_2)}
& \frac{1}{4}
\end{pmatrix} 
 M_2(t_{1}) \\
H^{\tilde{T}}_{2}
&=M_2(t_{2})^{-1} 
\begin{pmatrix}
-\frac{1}{4} & 0 \\
\frac{\lambda_0(a-1)}{2(\tilde{b} +2 a)}
+\frac{\lambda_1(a-1)}{2(\tilde{b}+2)}
+ \sum^{n-2}_{i=1}
\frac{ a^2 \alpha^i_0(t_2)}{b^2(a-1)^2(t_2 -t_1)}
&\frac{1}{4}
\end{pmatrix}
M_2(t_{2}).
\end{aligned}
\end{equation*}
Let $H^{\tilde{T}}_{2i+1}$ and $H^{\tilde{T}}_{2i+2}$ be
the residue matrices at $\tilde{x}=t_{2i+1}$ and $\tilde{x}=t_{2i+2}$, respectively.
We have the following equalities
\begin{equation*}
\begin{aligned}
H^{\tilde{T}}_{2i+1}&=
M_2(t_{2i+1})^{-1}
\begin{pmatrix}
0 & \frac{ \lambda_{i+1} (a d_i -b c_i t_{2i+1}) }{{2a}} \\
\frac{ \lambda_{i+1} a(-a d_i +b c_i t_{2i+1})}{{2b^2(a-1)^2 (t_{2i+1} -t_1)(t_{2i+1} -t_2)}} &0
\end{pmatrix}
M_2(t_{2i+1})\\
H^{\tilde{T}}_{2i+2}&=
M_2(t_{2i+2})^{-1}
\begin{pmatrix}
0 & \frac{ \lambda_{i+1}   (a d_i -b c_i t_{2i+2}) }{{2a }} \\
\frac{ \lambda_{i+1}  a (-a d_i +b c_i t_{2i+2})}{{2b^2(a-1)^2 (t_{2i+2} -t_1)(t_{2i+2} -t_2)}} &0
\end{pmatrix}
M_2(t_{2i+2}).
\end{aligned}
\end{equation*}
Let $H^{\tilde{T}}_{2n+1}$ be
the residue matrix at $\tilde{x}= \infty$.
Let $\cA_{jk}^i (\tilde{x})$ be 
the relative rational $1$-forms over $\tilde{T}$ which are the relativization of
the pull-backs of the rational $1$-forms (\ref{2018.5.21.12.20}) under
the composition $\mathbb{P}^1 \times \tilde{T} \ra\mathbb{P}^1 \times T \ra \mathbb{P}^n$.
Since $\lim_{\tilde{x}\ra 0 } \alpha_0^i(\frac{1}{\tilde{x}})=0$, we have
\begin{equation*}
\res_{\tilde{x}=\infty}
\begin{pmatrix}
\cA_{11}^i (\tilde{x}) & \cA_{12}^i (\tilde{x}) \\
-\cA_{21}^i (\tilde{x}) &- \cA_{11}^i (\tilde{x})
\end{pmatrix}=0 \quad (i=1,\ldots,n-2).
\end{equation*}
Then we have
\begin{equation*}
H_{2n+1}^{\tilde{T}}=
\begin{pmatrix}
-1 & a -2 \\
1&a 
\end{pmatrix}
\begin{pmatrix}
\frac{\lambda_0 + \lambda_1}{2} &0 \\
0&-\frac{\lambda_0 + \lambda_1}{2}
\end{pmatrix}
\begin{pmatrix}
-1 & a -2 \\
1&a 
\end{pmatrix}^{-1}.
\end{equation*}

%%%%%%%%%%%%%%%%%%%%%%%%%%%%%%%%%%%%%%%%%%%%%%%%
\subsection{Garnier system}\label{2019.7.10.22.15}
%%%%%%%%%%%%%%%%%%%%%%%%%%%%%%%%%%%%%%%%%%%%%%%%

%Following \cite{Mazz}, we recall the Garnier systems.
Let $\cA(\tilde{x})$ be the Fuchsian system with $2n + 1$ regular singularities 
at $t_1 , \ldots , t_{2n}, \infty$:
\begin{equation*}
\cA(\tilde{x})= d+ 
\sum_{i=1}^{2n} \tilde{H}_{i}\frac{d\tilde{x}}{\tilde{x}-t_i} ,
\end{equation*}
where $\tilde{H}_{i}$ ($i=1, \ldots , 2n$) are 
$2\times 2$ matrices independent of $\tilde{x}$ and 
$t_i \neq t_j$ ($i \neq j$). 
We assume that  $\tilde{H}_{2n+1}:=-\sum_{i=1}^{2n} \tilde{H}_{i}$ is a diagonal matrix and 
the eigenvalues of $\tilde{H}_{i}$ ($i=1, \ldots ,2n+1$) are
as in Table \ref{2018.5.20.22.24}.
\begin{table}[htb]
\caption{The eigenvalues of the residue matrices 
($i=1,\ldots , n-2$).}
  \begin{tabular}{c|ccccccc}\label{2018.5.20.22.24}
  Reside matrices  & $\tilde{H}_{1}$ &
    $\tilde{H}_{2}$ & $\tilde{H}_{2i+1}$ & $\tilde{H}_{2i+2}$
    &  $\tilde{H}_{2n-1}$ &  $\tilde{H}_{2n}$ & $\tilde{H}_{2n+1}$ \\\hline
  Eigenvalues 
     & $\pm \frac{1}{4}$ 
      & $\pm \frac{1}{4}$ 
      & $\pm \frac{\lambda_{i+1}}{2}$ 
      & $\pm \frac{\lambda_{i+1}}{2}$ 
       & $\pm \frac{\lambda_1}{2}$  
    & $\pm \frac{\lambda_0-1}{2}$ 
      &  $\pm\frac{\lambda_0+\lambda_1}{2}$
  \end{tabular}
\end{table}

We fix generators $\gamma_{\tilde{x}}$ 
($\tilde{x}= t_1 \ldots ,t_{2n}, \infty$)
of the fundamental group $\pi_1(\mathbb{P}^1 \setminus \{ t_1,\ldots,t_{2n}, \infty \},*)$.
Here the loop $\gamma_{\tilde{x}}$ on $\mathbb{P}^1$ 
is oriented counter-clockwise, 
$\tilde{x}$ lies inside, while the other singular points lie outside.
Let $\rho'_{\boldsymbol{\lambda}}\colon 
\pi_1(\mathbb{P}^1 \setminus\{ t_1,\ldots,t_{2n}, \infty \},*) \ra \mathrm{SL}_2(\mathbb{C})$
be the representation of the fundamental group defined by Table \ref{2018.5.12.23.13}.
We consider the isomonodromic deformation of the Fuchsian system $\cA(\tilde{x})$
whose preserved monodromy representation is conjugated to
 $\rho'_{\boldsymbol{\lambda}}$. 
Let
$d+ 
\sum_{i=1}^{2n} \tilde{H}^0_{i}\frac{d\tilde{x}}{\tilde{x}-t^0_i} $
be the Fuchsian system with $2n + 1$ regular singularities 
at $t^0_1 , \ldots , t^0_{2n}, \infty$ whose monodromy representation is
conjugated to $\rho'_{\boldsymbol{\lambda}}$.
There exists an open neighbourhood $U_{t^0}\subset \mathbb{C}^{2n}$ 
of the point $t^0=(t^0_1 , \ldots , t^0_{2n})$ 
such that for any $t \in U_{t^0}$, there exists a unique tuple 
$(\tilde{H}_{i}(t))_{i=1,\ldots,2n}$ of analytic matrix valued functions 
such that $\tilde{H}_{i}(t^0) = \tilde{H}^0_{i}$, $i=1,\ldots,2n$,
and the monodromy representation of 
 $d+ 
\sum_{i=1}^{2n} \tilde{H}_{i}(t)\frac{d\tilde{x}}{\tilde{x}-t_i} $
 is
conjugated to $\rho'_{\boldsymbol{\lambda}}$.
The matrices $\tilde{H}_{i}(t)$ $i=1,\ldots,2n$ are 
the solutions of the Cauchy problem 
with the initial data $(\tilde{H}_{i}(t^0))_{i=1,\ldots,2n}$ 
for the Schlesinger equations (see \cite[Theorem 2.7]{Mazz}).

\begin{table}[htb]
\caption{The representation $\rho'_{\boldsymbol{\lambda}}$ of the fundamental group;
here $a_j = \exp(-\pi \sqrt{-1} \lambda_j)$ $j=0,1,\ldots,n-1$.}
  \begin{tabular}{c|c|c|c}\label{2018.5.12.23.13}
     $\tilde{x}=t_{2n-1}$ &  $\tilde{x}=t_{2n}$ & $\tilde{x}=t_1$ &
    $\tilde{x}=t_2$ \\\hline
$ \rho'_{\boldsymbol{\lambda}}(\gamma_{t_{2n-1}}) 
 = \begin{pmatrix} a_1 & 0 \\ 0 & a_1^{-1} \end{pmatrix}$  
&  $\rho'_{\boldsymbol{\lambda}}(\gamma_{t_{2n}}) 
=\begin{pmatrix} -a_0 & 0 \\ 0 & -a_0^{-1} \end{pmatrix}$ 
& $\rho'_{\boldsymbol{\lambda}}(\gamma_{t_1})=\begin{pmatrix} 0 & 1 \\ -1 &0 \end{pmatrix}$ 
& $\rho'_{\boldsymbol{\lambda}}(\gamma_{t_2})
=\begin{pmatrix} 0 &  a_0^{2} \\ - a_0^{-2} & 0 \end{pmatrix}$ 
  \end{tabular}\\
  \begin{tabular}{c|c|c}
     $\tilde{x}=t_{2i+1}$ ($i=1,\ldots,n-2$)
     &  $\tilde{x}=t_{2i+2}$ ($i=1,\ldots,n-2$)
      & $\tilde{x}=\infty$\\\hline
 $\rho'_{\boldsymbol{\lambda}}(\gamma_{t_{2i+1}})=
 \begin{pmatrix} a_{i+1} & 0 \\ 0 & a_{i+1}^{-1} \end{pmatrix}$  
 & $\rho'_{\boldsymbol{\lambda} }(\gamma_{t_{2i+2}}) 
 = \begin{pmatrix} a^{-1}_{i+1} & 0 \\ 0 & a_{i+1} \end{pmatrix}$ 
 & $\rho'_{\boldsymbol{\lambda}}(\gamma_{\infty})=
 \begin{pmatrix} a_0 a_1^{-1} & 0 \\ 0 &a_0^{-1} a_1 \end{pmatrix}$ 
  \end{tabular}
\end{table}

Let $\cA(\tilde{x})$ be the Fuchsian system with $2n + 1$ regular singularities 
at $t_1 , \ldots , t_{2n}, \infty$ as above.
We fix the poles $t_{2n-1}$ and $t_{2n}$ at $0$ and $1$, respectively.
Let $\{\nu_1,\ldots,\nu_{2n-2}\}$ be the roots of the following equation of degree $2n-2$:
\begin{equation}\label{2018.5.12.23.46}
\sum_{k=1}^{2n} \frac{(\tilde{H}_{k})_{12}}{\tilde{x}-t_k} =0.
\end{equation}
For each $\nu_i$, we define $\rho_i$ by
\begin{equation}\label{2018.5.12.23.47}
\begin{aligned}
\rho_i := 
\sum_{k=1}^{2n}  \frac{(\tilde{H}_{k})_{11}+\frac{\theta_{k}}{2}}{\nu_i-t_{k}}.
\end{aligned}
\end{equation}
If a tuple $(\tilde{H}_{i}(t))_{i=1,\ldots,2n}$ is 
a solution of the Schlesinger equations,
then the corresponding functions $\nu_{j}(t_1,\ldots,t_{2n-2})$ and 
$\rho_{j}(t_1,\ldots,t_{2n-2})$ ($j=1,\ldots,2n-2$)
satisfy the Garnier system $\mathcal{G}_{2n-2}$ (see \cite[Theorem 2.1]{Mazz}).

%%%%%%%%%%%%%%%%%%%%%%%%%%%%%%%%%%%%%%%%%%%%%%%%%
\subsection{Algebraic solution}\label{2019.7.12.11.28}
%%%%%%%%%%%%%%%%%%%%%%%%%%%%%%%%%%%%%%%%%%%%%%%%%
By the morphism (\ref{2018.5.12.23.51}), we have a generically finite morphism
\begin{equation*}
\Spec \mathbb{C}[\rho_i,\nu_i]_{1\le  i  \le2n-2} \times \tilde{T}\lra 
\Spec \mathbb{C}[\rho_i,\nu_i]_{1\le  i  \le2n-2} \times \Spec \mathbb{C}[t_1,\ldots,t_{2n-2}].
\end{equation*}
%Let $\mathcal{G}^{\tilde{T}}_n$ be the induced system by the Garnier system $\mathcal{G}_n$
%under the morphism (\ref{2018.5.12.23.54}).
We consider the algebraic solution of $\mathcal{G}_{2n-2}$ associated to 
the representation $\rho'_{\boldsymbol{\lambda}}$.

\begin{proof}[Proof of Theorem \ref{2019.7.10.22.13}]
For the residue matrices $H^{\tilde{T}}_{i} $ of 
$(\nabla_{\mathbb{P}^1 \times \tilde{T}/\tilde{T}})_{\boldsymbol{\lambda}}$,
we put
\begin{equation*}
\tilde{H}^{\tilde{T}}_{i}:= 
\begin{pmatrix}
-1 & a -2 \\
1&a 
\end{pmatrix}^{-1}
H^{\tilde{T}}_{i} 
\begin{pmatrix}
-1 & a -2 \\
1&a 
\end{pmatrix}
\end{equation*}
for $i = 1, \ldots, 2n$.
Let $\cA^{\tilde{T}}(\tilde{x})$ be the family of the Fuchsian systems  
with $2n + 1$ regular singularities at $0,1,t_1 , \ldots , t_{2n-2}, \infty$ 
parametrized by $\tilde{T}$
defined by
\begin{equation}\label{2018.5.12.23.49}
\cA^{\tilde{T}}(\tilde{x}):=d+
\tilde{H}^{\tilde{T}}_{2n-1}\frac{d\tilde{x}}{\tilde{x}} + 
\tilde{H}^{\tilde{T}}_{2n}\frac{d\tilde{x}}{\tilde{x}-1} +
\sum_{i=1}^{2n-2} \tilde{H}^{\tilde{T}}_{i}\frac{d\tilde{x}}{\tilde{x}-t_i} .
\end{equation}
Since $\tilde{H}^{\tilde{T}}_{2n+1}:=
-\sum_{i=1}^{2n} \tilde{H}^{\tilde{T}}_{i}$ 
is a diagonal matrix,
we have the assertion (i) of Theorem \ref{2019.7.10.22.13}.

By Proposition \ref{2018.5.8.15.13},
for each $\tilde{t} \in \tilde{T}$, 
the Fuchsian system $\cA^{\tilde{T}}(\tilde{x})$
has the monodromy representation which is conjugated to $\rho'_{\boldsymbol{\lambda}}$,
which is independent of $\tilde{t} \in\tilde{T}$.
That is, the family $\cA^{\tilde{T}}(\tilde{x})$ 
of the Fuchsian systems parametrized by $\tilde{T}$
preserves 
their monodromy representations.
Then we have the assertion (ii) of Theorem \ref{2019.7.10.22.13}.

By (\ref{2018.5.12.23.46}), (\ref{2018.5.12.23.47}), and (\ref{2018.5.12.23.49}),
we have algebraic functions $\nu_i$, $\rho_i$ ($i=1,\ldots,2n-2$) on $\tilde{T}$.
These algebraic functions give the solution of $\mathcal{G}_{2n-2}$ associated to 
the representation $\rho'_{\boldsymbol{\lambda}}$.
Since $\dim \tilde{T} =2n-2$ and 
$\tilde{T} \ra \Spec \mathbb{C}[t_1,t_2,\ldots ,t_{2n-2}]$
is a generically finite morphism, 
we have the assertion (iii) of Theorem \ref{2019.7.10.22.13}.
\end{proof}

%%%%%%%%%%%%%%%%%%%%%%%%%%%%%%%%%%%%%%%%%%%%%%%%%%
%%%%%%%%%%%%%%%%%%%%%%%%%%%%%%%%%%%%%%%%%%%%%%%%%%

\noindent
{\bf Acknowledgments.}
The author thanks Frank Loray for many valuable discussions.
He also thanks Masa-Hiko Saito for warm encouragement.
He is supported by JSPS KAKENHI Grant Numbers 18J00245 
and 19K14506.
He is grateful to the anonymous referee's suggestions which helped to improve the paper.

%%%%%%%%%%%%%%%%%%%%%%%%%%%%%%%%%%%%%%%%%%%%%%%%%%
%%%%%%%%%%%%%%%%%%%%%%%%%%%%%%%%%%%%%%%%%%%%%%%%%%

\end{document}